\documentclass{article}
\usepackage[english]{babel}
\usepackage[T1]{fontenc}
\usepackage[utf8]{inputenc}
\usepackage{amsthm}
\usepackage{enumerate}
\usepackage{amsfonts}
\usepackage{verbatim}
\usepackage{graphicx}
\usepackage[babel=true]{csquotes}
\usepackage{amsmath}
\usepackage[left=2cm,right=2cm,top=3cm,bottom=3cm]{geometry}
\usepackage{amssymb}
\usepackage{mathrsfs}
\usepackage{compactbib}
\usepackage{mathtools}
\newtheorem{defi}{Definition}[subsection]
\newtheorem{theo}{Theorem}[subsection]
\newtheorem*{theo2}{Theorem}
\newtheorem{prop}{Proposition}[subsection]

\newtheorem{lem}{Lemma}[subsection]
\newtheorem{coro}{Corollary}[subsection]

\title{$\alpha$-numbers, diophantine exponent and factorisations of sturmian words}
\author{Ca\"{\i}us Wojcik\footnote{Contact : caius.wojcik@gmail.com}}
\date{}

\begin{document}

\maketitle

\begin{abstract}
We introduce the notion of $\alpha$-numbers and formal intercept of sturmian words, and derive from this study general factorisations formula for sturmian words. Sturmian words are defined as infinite words with lowest unbound complexity, and are characterized by two parameters, the first one being well-known as the slope, and the second being their formal intercepts. We build this formalism by a study of Rauzy graphs of sturmian words, and we use this caracterisation to compute the repetition function of sturmian words and their diophantine exponent. We then develop these techniques to provide general factorisations formulas for sturmian words.

\end{abstract}

\bigskip


\bigskip

Sturmian words are the infinite words with lowest unbounded complexity, and are linked to diophantine approximation through the continued fraction expansion of the proportion of the letter 1 appearing in these infinite words, this first parameter describing sturmian words being called the slope. In this paper, we give a combinatorial bijective description of these words by the use of the Ostrowski numeration system associated to a slope, to define the notion of $\alpha$-numbers, and study some implications of this formalism. The paper is organised with three parts : in the first one we provide lemmas and structure description for the Rauzy graph of sturmian words, and study the basic properties of their repetition function. In the second part we give the definition and main theorem about $\alpha$-numbers and formal intercepts of sturmian words, and give some application of this description to compute the repetition function and the diophantine exponent of sturmian words. In the third part we study the implication of the notion of $\alpha$-numbers to give precise conditions of existence and unicity on the factorisations of sturmian words, along with some constructions on formal intercepts.

Given an irrational number $\alpha=[0,a_1, a_2,a_3,\ldots]$ expanded through its continued fraction, the Ostrowski numeration system consists of finite sums of the form $\sum_{i}b_{i+1}q_i$, where $(q_i)_{i\geq -1}$ is the sequence of denominators of the partial fractions of the continued fraction, and where the $(b_i)$ satisfy the so-called Ostrowski conditions, these sums being known as a bijective description of integers. An $\alpha$-number $\rho$ of the slope $\alpha$ is then given by a formal sum
\begin{center}
	$\displaystyle \rho = \sum_{i= 0}^{+\infty}b_{i+1}q_i $
\end{center}
where the coefficients $(b_i)_{i\geq 1}$ satisfy the Ostrowski conditions, and therefore are not necessarily eventually zero. Our main result, proved in section 2 along with the definitions of $\alpha$-numbers, is the following, where $T$ is the shift operator on infinite words :
\begin{theo2}
Every sturmian word of slope $\alpha$ with continued fraction expansion $[0,a_1,a_2,a_3,\ldots]$ writes uniquely in the form :
\begin{center}
	$x=T^{\rho}(c_\alpha)$
\end{center}
where $\rho$ is an $\alpha$-number and $c_\alpha$ the characteristic word, both with respect to the slope $\alpha$.
\end{theo2}

With this result we give a general formula for the diophantine exponent of sturmian words, and provide some general results concerning factorisations of sturmian words.

\section{General context, repetition function and Rauzy graphs of sturmian words}

\subsection{Complexity function and characteristic sturmian word}

For an infinite word $x=x_1x_2x_3\ldots\in \mathcal{A}^{\mathbb{N}}$ over a finite alphabet $\mathcal{A}$, we define the complexity function $p(x,\cdot )$ as	$p(x,n)=Card\{x_ix_{i+1}\cdots x_{i+n-1}\ | \ i\geq 1\}$ for $n\geq 1$, that is, $p(x,n)$ is the number of factors of length $n$ appearing in $x$.


For $x$ an infinite word over an alphabet $\mathcal{A}$, the Morse and Hedlund theorem states that the word $x$ is ultimately-periodic if and only if there is some $n\geq 1$ such that $p(x,n)\leq n$, see \cite{morsehedlund}, theorem 7.3. for the original proof. As a consequence of this theorem, it follows that a non-ultimately periodic word $x$ satisfies the inequality $\forall n \geq 1, p(x,n)\geq n+1$.

\begin{defi}\label{sturmiendef}
An infinite word $x$ is said to be sturmian if
\begin{center}
	$\forall n\geq 1$, $p(x,n)=n+1$.
\end{center}
\end{defi}

Note that a sturmian word $x$ must be defined over a 2-letter alphabet since $p(x,1)=2$, and for the remainder of the paper we will assume that $\mathcal{A}=\{0,1\}$. 

For $x$ an infinite word over $\mathcal{A}=\{0,1\}$, $x$ is sturmian, if and only if, for every factor $u$ and $v$ of $x$ with $|u|=|v|$, we have $||u|_1-|v|_1|\leq 1$, and the number $\alpha= \lim_{|u|\rightarrow +\infty} \frac{|u|_1}{|u|}$ is irrational. This irrational number $0<\alpha<1$ is called the slope of the sturmian word $x$. For the remainder of the paper, we will call a slope any irrational number $\alpha$ with $0<\alpha<1$.

For an infinite word $x$, we define the dynamical orbit (also known as the subshift) of $x$ as the set :
\begin{center}
	$\Omega(x)=\overline{\{T^k(x) \ | \ k\geq 0\}}$
\end{center}
that is the topological closure of the set of suffixes of $x$, with the set $\mathcal{A}^\mathbb{N}$ endowed with the product topology associated to the discret topology on $\mathcal{A}$. We recall that $T : x=x_1x_2\ldots \mapsto T(x)=x_2x_3\ldots$ is the shift on infinite words. For an infinite word $x$, we denote by $\mathbb{P}_n(x)$ its prefix of length $n\geq 1$.

It is known that two sturmian words of same slope have the same set of factors and the same dynamical orbit, and conversely two sturmian words of different slopes only share a finite number of factors and have disjoint dynamical orbit (see for example Chap. 2 of \cite{lothaire}).

From the definition of sturmian words, it is easy to see that for $x$ a sturmian word and all $n\geq 1$, there is exactly one factor $L_n$ of length $n$ such that both $0L_n$ and $1L_n$ are factors of $x$. This factor is called a left special factor, and we define similarly the right special factors as factors $R_n$ of length $n$ such that both $R_n0$ and $R_n1$ are factors of $x$. By unicity, the factors $L_n$ for $n\geq 1$ are prefixes of one another, and the factors $R_n$ are suffixes of one another.

\begin{defi}\label{caracdef}
For every sturmian word of slope $\alpha$, the sequence $(L_n)_{n\geq 1}$ of its left special factors defines an infinite word :
\begin{center}
	$c_\alpha=\lim L_n$
\end{center}
that depends only on $\alpha$, and is called the characteristic word of the slope $\alpha$.
\end{defi}

For a finite word $u$, we denote by $\widetilde{u}$ the reversal (or mirror) or $u$, and we say that $u$ is palindromic if $u=\widetilde{u}$. The characteristic word of slope $\alpha$ satisfies the following properties : 1) a sturmian word $x$ is characteristic if and only if $0x$ and $1x$ are both sturmian, 2)  $\forall n\geq 1$, $R_n=\widetilde{L_n}$, 3)  the set of factors of a sturmian word is stable under reversal, 4)  for every sturmian word $x$, at least one of the two words $0x$ and $1x$ is sturmian.

The characteristic word of a slope $\alpha$ can be constructed with the help of the continued fraction expansion of the number $\alpha$ as follows. Recall that $\alpha$ writes uniquely as 
	\begin{center}
	$\alpha =[0;a_1,a_2,\ldots]= \cfrac{1}{a_1+\cfrac{1}{a_2+\cfrac{1}{a_3+\ldots}}}$
\end{center}
where the $(a_i)_{i\geq 1}$ are positive integer, called the partial quotients of $\alpha$. We define the sequence $(q_n)_{n\geq -1}$ of positive integers as the denominators of the irreducible fraction
\begin{center}
$ \frac{p_n}{q_n}=[0;a_1,\ldots,a_n]= \cfrac{1}{a_1+\cfrac{1}{\ldots+\frac{1}{a_n}}}$
\end{center}
with the additional values $q_{-1}=0$ and $q_0=1$, this sequence $(q_n)_{n\geq -1}$ is the sequence of continuants of $\alpha$ and satisfies the foundamental recurrence relation
\begin{center}
	$q_{n+1}=a_{n+1}q_n+q_{n-1}$
\end{center}
for every $n\geq 0$.

The characteristic sturmian word of slope alpha is then obtained as the limit of words $c_\alpha=\lim s_n$ where the sequence $(s_n)_{n\geq -1}$ of finite words is defined by the recurrence $s_{-1}=1$, $s_0=0$, $s_1=s_0^{a_1-1}s_{-1}$, $s_{n+1}=s_n^{a_{n+1}}s_{n-1}$ for all $n\geq 1$.
Although this construction of the sturmian characteristic word could be used as a definition of characteristic sturmian words, which would allow to define sturmian words in general, it is a quite non-trivial construction from the definitions we used of sturmian words. Proofs and details about this construction can be found in the classical reference \cite{lothaire}, theorem 2.1.5. This construction leads to constructibility and decidability problems on characteristic words, see for example the recent paper \cite{durand2} on this topic.

The study of sturmian words uses the combinatorial properties of the two classes of finite words called standard words and central words, who are closely related to each other. Standard words and central words are connected to Christoffel words, see \cite{berthe4, labbe} in this topic. The set of standard pairs, which is a subset of $(\mathcal{A}^*)^2$ is recursively defined by the rules : $(0,1)$ is a standard pair, and if $(u,v)$ is a standard pair, then $(vu,v)$ and $(u,uv)$ are standard pairs. A word is said to be standard if it is a component of a standard pair. We use the notation that $u^{-}$ denotes a finite word $u$ whose last letter is removed (with the obvious convention for the empty word). For $(u,v)$ a standard pair, we have $(uv)^{--}=(vu)^{--}$, if $|u|\geq 2$, then $u$ ends with $10$, and if $|v|\geq 2$, then $u$ ends with $01$. The sequence $(s_n)_{n\geq -1}$ appearing in the construction of a characteristic sturmian word is made of standard words. In particular, if $n\geq 1$ is even, $s_n$ ends with $10$, and if $n\geq 2$ is odd, $s_n$ ends with $01$. The closely related class of central words can be defined as words of the form $u^{--}$ for $u$ a standard word, palindromic prefixes of characteristic sturmian words, or equivalently as powers of letters or palindromes of the form $p01q$ where $p$ and $q$ are palindromes, the latter decomposition being unique.

\subsection{repetition function and Rauzy graph of sturmian words}

In this section we present the definitions of Rauzy graphs (also known as factor graphs) and of the repetition function of infinite words. Rauzy graphs have been introduced by G. Rauzy along with a study of arithmetic sequences \cite{rauzy}. This complicated mathematical object encodes deep properties of infinite words, it is for example remarkable that in the case of sturmian words, even if this graph is quite simple from the low complexity condition on sturmian words, it is made of two cycles with length that are relatively prime. See also \cite{arnoux2, avgustinovich2, cassaigne4}.

\begin{defi}
For $x$ an infinite word over a finite alphabet $\mathcal{A}$, and $m\geq 1$, we define the Rauzy graph $G_m$ of degree $m\geq 1$ of $x$ as the directed graph with :
\begin{itemize}
	\item vertexes as factors of $x$ with length $m$,
	\item two factors $s$ and $t$ are connected by a directed arrow $s\rightarrow t$ when there is a factor $w$ of $x$ of length $m+1$ having $s$ as prefix and $t$ as suffix.
\end{itemize}
\end{defi}

For every $m\geq 1$ and every infinite word $x$, the word $x$ defines a path in its Rauzy graph $G_m$ of degree $m$ through
\begin{center}
	$\mathbb{P}_m(x)\longrightarrow \mathbb{P}_m(T(x)) \longrightarrow \mathbb{P}_m(T^2(x))\longrightarrow \cdots \longrightarrow \mathbb{P}_m(T^k(x)) \longrightarrow \cdots $
\end{center}
where we recall the notation $\mathbb{P}_m(x)$ for a prefix of length $m$ of $x$. In the case of sturmian words, the Rauzy graph $G_m$ of $x$ of degree $m$ consists of the fusion of two cycles sharing a common path. This is due to the fact that for a given length $m$, there is exactly one left special factor in $x$ of length $m$ corresponding to a vertex with in-degree 2, as well as exactly one right special factor of length $m$ corresponding to a vertex with out-degree 2. This structure is linked to the so-called three-gap theorem, see \cite{berthe2}.

We give now the definition of the repetition function associated to an infinite word $x$. This function is considered as a complexity function, and has been introduced independently by Y. Bugeaud and D. Kim \cite{bugeaudkim1} on the one hand, and S. Moothathu \cite{moothathu} on the other hand, and another study can be found in \cite{rampersad}. However, our repetition function is slightly different from the function introduced by Y. Bugeaud and D. Kim : if the latter is denoted by $r_0(x,m)$ then it is linked to our own repetition function by the relation $r_0(x,m)=r(x,m)+m$ for all $m\geq 1$.

\begin{defi}\label{repetitiondef}
For $x$ an infinite word over a finite alphabet $\mathcal{A}$, we set for $m\geq 1$, $r(x,m)$ as the largest integer $k\geq 1$ such that the words \begin{center}
$\mathbb{P}_m(x), \  \mathbb{P}_m(T(x)),  \  \ldots ,  \   \mathbb{P}_m(T^{k-1}(x))$
\end{center}
are all pairwise distincts. The function $r(x,\cdot)$ is called the repetition function of $x$.
\end{defi}

For a general word $x$, we have the inequality $r(x,n)\leq p(x,n)$ for all $n\geq 1$, which becomes $r(x,n)\leq n+1$ in the case of sturmian words. We give a proof of the following statement relying on the use of Rauzy graphs, although this result can be easily derived from the work of Bugeaud and Kim \cite{bugeaudkim1}.

\begin{prop}\label{repetitioncomportement}
For $x$ a sturmian word and $m\geq 2$, the following statements are equivalent :
\begin{enumerate}[i)]
	\item $r(x,m)=m+1$,
	\item $r(x,m)\neq r(x,m-1)$.
\end{enumerate}
\end{prop}

\begin{proof}
The implication $(i)\Rightarrow(ii)$ is clear since $r(x,m-1)\leq m$. For the converse, let $A_m$ and $B_m$ be the two vertexes of $G_m$ such that	$R_m\rightarrow A_m$ and $R_m\rightarrow B_m$
in $G_m$. We consider the path
\begin{center}
	$\mathbb{P}_{m-1}(x)\rightarrow \mathbb{P}_{m-1}(T(x))\rightarrow \cdots \rightarrow \mathbb{P}_{m-1}(T^{r(x,m-1)}(x))$.
\end{center}
in $G_{m-1}$. There is a unique integer $0\leq j < r(x,m-1)$ such that $\mathbb{P}_{m-1}(T^{r(x,m-1)}(x))=\mathbb{P}_{m-1}(T^j(x))$.

 In $G_m$, we cannot have $\mathbb{P}_{m}(T^{r(x,m-1)}(x))=\mathbb{P}_{m}(T^j(x))$ for that would imply $r(x,m)= r(x,m-1)$, in contradiction with our hypothesis. Hence we have $\mathbb{P}_{m}(T^{r(x,m-1)}(x))\neq\mathbb{P}_{m}(T^j(x))$ and those two words differ only by their last letters. This shows the equality of sets $\{A_m,B_m\}=\{\mathbb{P}_{m}(T^{r(x,m-1)}(x)),\mathbb{P}_{m}(T^j(x))\}$ so that the path
\begin{center}
	$\mathbb{P}_{m}(x)\rightarrow \mathbb{P}_{m}(T(x))\rightarrow \cdots \rightarrow \mathbb{P}_{m}(T^{r(x,m)-1}(x))$
\end{center}
goes through the two vertexes $A_m$ and $B_m$. This path is the longest hamiltonian path starting at the vertex $\mathbb{P}_{m}(x)$ in the path defined by $x$, and with in mind the structure of the Rauzy graph of sturmian words, this path has to go through every of the $m+1$ vertexes of $G_m$, giving the value $r(x,m)=m+1$.
\end{proof}

The next result shows that, in the case of characteristic sturmian words, the first repetition of a factor has to come from a prefix.

\begin{lem}\label{repetitioncarac}
For the characteristic sturmian word $c_\alpha$ of the slope $\alpha$, we have :
\begin{center}
	$\mathbb{P}_{m}(T^{r(c_\alpha,m)}(c_\alpha))=\mathbb{P}_m(c_\alpha)=L_m$
\end{center}
for all $m>0$.
\end{lem}

\begin{proof}
The equality on the right comes from the definition of $c_\alpha$ as the limit of its left special factors. For the equality on the left, let $0\leq j < r(c_\alpha,m)$ be the unique integer such that $\mathbb{P}_{m}(T^{r(c_\alpha,m)}(c_\alpha))=\mathbb{P}_{m}(T^j(c_\alpha))$ and assume by contradiction that $j\neq 0$. Then $\mathbb{P}_{m}(T^{r(c_\alpha,m)-1}(x))\neq\mathbb{P}_{m}(T^{j-1}(c_\alpha))$ and those two words only differ by their last letters. Hence the word $\mathbb{P}_{m}(T^j(c_\alpha))$ is left special, which leads to $j=0$ and the desired contradiction.
\end{proof}

We extend the definition of the repetition function to finite words $z$, such that $z$ has a factor of length $m$ having two occurrences in $z$, as the value $r(x,m)$ of any infinite word $x$ having $z$ as a prefix.

\begin{lem}\label{repetitioncentraux}
Let $z=p01q$ be a central word with $|p|\leq |q|$, where $p$ and $q$ are palindromic. Then :
\begin{center}
	$r(z,|p|+1)=|p|+2$.
\end{center}
\end{lem}

\begin{proof}
Let $c_{\alpha}$ be a characteristic sturmian word having $z$ as a prefix. We know that $\mathbb{P}_{|p|+1}(T^{r(z,|p|+1)}(c_\alpha))=p0$.

We prove the result by induction on $|z|$. Since $z$ is palindromic, we cannot have $|p|=|q|$, and if $|p|=|q|-1$ then $q=p0=0p$, and hence $z=0^{|p|+1}10^{|p|+1}$, which closes the discussion in this case. We now assume that $|p|\leq|q|-2$ and we set $q=p01u$ for a finite word $u$. The word $u$ is palindromic since $z=q10p=p01u10p$ is palindromic, so that $q=p01u$ is the unique factorisation of $q$ as a central word.

If $|p|\leq |u|$, then the result is derived by induction hypothesis. Assuming otherwise that $|u|\leq |p|$ leads to $r(z,|u|+1)=r(q,|u|+1)=|u|+2$ by induction hypothesis. Since $r(z,|u|+1)\leq r(z,|p|) \leq |u|+2$, we must have $r(z,|p|)=|u|+2$. But $z=u10p10p$, so that the word $u10p0$ is not a prefix of $z$, and we deduce that
\begin{center}
	$r(z,|p|+1)>r(z,|u|+1)=|u|+2=r(z,|p|)$
\end{center}
along with $r(z,|p|+1)\neq r(z,|p|)$. The induction ends as a consequence of proposition \ref{repetitioncarac}. \end{proof}

\begin{coro}
For $c_\alpha$ the characteristic word of slope $\alpha$, with continuants $(q_n)_{n\geq -1}$, we have :
\begin{center}
	$r(c_\alpha,m)=q_n$ \quad for \quad $q_n-1\leq m \leq q_{n+1}-2$.
\end{center}
\end{coro}

\begin{proof}
Let $(s_n)_{n\geq -1}$ be the standard sequence associated with $\alpha$ upon the construction of $c_\alpha=\lim s_n$. It is clear that $|s_n|=q_n$ for $n\geq 0$. We have
\begin{center}
	$c_\alpha =\displaystyle \lim_{n\geq 1} s_{n+2}= \lim_{n\geq 1} s_{n+1}s_n=\lim_{n\geq 1} s_{n+1}s_n^{--}=\lim_{n\geq 1} s_n^{--}t_ns_{n+1}^{--}$
\end{center}
where $t_n=10$ if $n$ is even, and $t_n=01$ if $n$ is odd. The words $s_n^{--}t_ns_{n+1}^{--}$ for $n\geq 2$ are the central prefixes of $c_\alpha$ written in their central factorisations. By lemma \ref{repetitioncentraux}, we have :
\begin{center}
	$r(c_\alpha,|s_n^{--}|+1)=|s_n^{--}|+2=|s_n|=q_n$
\end{center}
and since the prefix $s_{n+1}^{--}$ of $c_{\alpha}$ has a repetition at the index position $|s_n^{--}|+2$ in $c_\alpha$, we derive $r(c_\alpha,m)=q_n$ for $n\geq 2$ and $q_n-1\leq m \leq q_{n+1}-2$.
\end{proof}

\subsection{Rauzy graphs of sturmian words}

For the remainder of the paper, intervals $[a,b]$ are considered as integer interval, meaning that they denote the set of integers $k$ with $a\leq k\leq b$. We define the integer intervals $I_n$, for $n\geq 0$,
	\begin{center}
		$I_n=[q_n-1,q_{n+1}-2]$,
		
    $I_n^0=[q_n-1,q_n+q_{n-1}-2]$,
	\end{center}
and for $1\leq l \leq a_{n+1}-1$,
\begin{center}
	$I_n^l=[lq_n+q_{n-1}-1,(l+1)q_n+q_{n-1}-2]$.
\end{center}
They form a partition of $\mathbb{N}^*$ :
\begin{center}
	$\displaystyle \mathbb{N}^*=\bigcup _{n\geq 0}I_n= \bigcup _{n\geq 0}\bigcup_{l=0}^{a_{n+1}-1}I_n^l$
\end{center}
so that every natural integer $m\geq 1$ writes uniquely as $m=\max I_n^l -r=(l+1)q_n+q_{n-1}-2-r$ where $n\geq 1$, $0\leq l \leq a_{n+1}-1$ and $0\leq r < |I_n^l|$, keeping in mind that $|I_n^0|=q_{n-1}$ and $|I_n^l|=q_{n}$ for $1\leq l \leq a_{n+1}-1$. If $a_1=1$ or $a_1=2$ then $I_0$ is empty. If $a_1=1$ and $a_2=1$, the two intervals $I_0$ and $I_1$ are empty. This partition of $\mathbb{N}$ encodes the structure of Rauzy graph of sturmian words, and we use our results on the repetition function to derive the length of the cycles of the Rauzy graph of sturmian words. This result can be derived from the work of V. Berthé \cite{berthe2} on the frequency of factors of sturmian words, however our study goes a little further since we give more precisions on the path taken by the characteristic word in the next statements. We use the notation $u^*$ to denote the word $u$ from which its first letter is removed, and we recall that $t_n=10$ if $n$ is even, and $t_n=01$ if $n$ is odd.

\begin{prop}\label{rauzystructure}
For $m\in I_n^l$ with $n\geq 0$ and $0\leq l \leq a_{n+1}-1$, we have :
\begin{enumerate}
	\item one of the two cycles of $G_m$ is of length $q_n$, it is called the referent cycle of $G_m$,
	\item the other cycle is of length $lq_n+q_{n-1}$,
	\item the arrow $R_m\rightarrow R_m^*t_{n-1}^{-}$ belongs to the referent cycle, and the arrow $R_m\rightarrow R_m^*t_{n}^{-}$ belongs to the non-referent cycle. None of these two arrows belong to the common part.
\end{enumerate}
\end{prop}

\begin{proof}
1) Since two infinite words having same set of factors share the same Rauzy graph, we can restrict to the case $x=c_\alpha$. Since we have seen that $r(c_\alpha,m)=q_n$, the path
\begin{center}
	$\mathbb{P}_{m}(c_\alpha)\rightarrow \mathbb{P}_{m}(T(c_\alpha))\rightarrow \cdots \rightarrow \mathbb{P}_{m}(T^{r(c_\alpha,m)}(c_\alpha))$
\end{center}
defines a cycle of length $q_n$ in $G_m$.

2) The common part of the two cycles of $G_m$ is the shortest path starting at the left special factor $L_m$ and ending at the right sepcial factor $R_m$. The finite words $w$ defined by this path is both left and right special and hence is the smallest central factor of $x$ with length $|w|\geq m$, and its length then equals $(l+1)q_n+q_{n-1}-2-m$. But since the sum of the lengths of the two cycles equals the number of vertexes of $G_m$ added with the number of element in the common part, we obtain, where $\mu$ is the length of the non-referent cycle :
\begin{center}
	$q_n + \mu = m+1 + (l+1)q_n+q_{n-1}-1-m $
\end{center}
so that the length of the non-referent cycle equals $\mu=lq_n+q_{n-1}$. Notice that the lengths of the two cycles are relatively prime, so that the referent cycles is well defined by its length.

3) The comon part $L_m\rightarrow \cdots \rightarrow R_m$ corresponds to the central word of length $(l+1)q_{n}+q_{n-1}-2$, which is the word $s_n^{l+1}s_{n-1}^{--}$, and the referent cycle corresponds to the path
\begin{center}
	$\mathbb{P}_{m}(c_\alpha)\rightarrow \mathbb{P}_{m}(T(c_\alpha))\rightarrow \cdots \rightarrow \mathbb{P}_{m}(T^{r(c_\alpha,m)}(c_\alpha))$
\end{center}
and it remains to show that $s_n^{l+1}s_{n-1}^{-}$ is a prefix of $c_\alpha$ since $s_{n-1}$ ends with $t_{n-1}$. But $s_{n-1}$ is a prefix of $s_n$ and $s_{n+1}=s_n^{a_{n+1}}s_{n-1}$, so that the arrow $R_m\rightarrow R_m^*t_{n-1}^{-}$ belongs to the referent cycle. The fact that the arrow $R_m\rightarrow R_m^*t_{n}^{-}$ belongs to the non-referent cycle comes from the fact that the word $s_n^{l+1}s_{n-1}^{--}t_n$ is not a prefix of $c_\alpha$. The last remark is clear since the two arrows coming out of the right special factor cannot both be on the same cycle. \end{proof}

Toward a more precise understanding of the path taken by the characteristic word on its Rauzy graph, we define formally what turning around a cycle means for an infinite word $x$ as follows : 
\begin{itemize}
	\item We say that $x$ turns around a cycle of length $k$ in $G_m$ when $r(x,m)=k$ and when $\mathbb{P}_m(x)\rightarrow \mathbb{P}_m(T(x))\rightarrow \cdots \rightarrow \mathbb{P}_m(T^k(x))$ shares the same arrows as the concerned cycle.
	\item We say that $x$ turns $d$ times around a cycle of length $k$ if, for every $1\leq i \leq d-1$, $T^{ik}(x)$ turns around this cycle.
\end{itemize}

\begin{theo}\label{caractournerauzy}
For $m\in I_n^l$, the characteristic word $c_\alpha$ turns $a_{n+1}-l$ times around the referent cycle, and no more.
\end{theo}

\begin{proof}
We first consider the case $l=0$. The central word $s_n^{--}$ is a strict prefix of $L_m$ and $L_m$ is a prefix of the central word $s_n^{--}t_ns_{n-1}^{--}$. The word $z=s_{n+1}s_n^{--}=s_n^{a_{n+1}+1}s_{n-1}^{--}$ is central and so we have
\begin{center}
	$r(z,m)=q_n=r(T^{q_n}(z),m)=\cdots=r(T^{(a_{n+1}-1)q_n}(c_\alpha),m),$
\end{center}
and this shows that $c_\alpha$ turns at least $a_{n+1}$ times around the referent cycle.

Since $s_{n+1}s_n$ is a prefix of $c_\alpha$, the word $s_{n+1}s_n=zt_n=s_n^{a_{n+1}+1}s_{n-1}^{--}t_n$ is a prefix of $c_\alpha$ and $s_ns_{n-1}^{--}t_n$ is a prefix of $T^{a_{n+1}q_n}(c_\alpha)$, and from there we see that the word $T^{a_{n+1}q_n}(c_\alpha)$ first goes through the arrow $R_m\rightarrow R_m^*t_{n}^{-}$, which do not belong to the referent cycle. This shows that $c_\alpha$ does not turn $a_{n+1}+1$ around the referent cycle.

The case $l>0$ is similar : the central word $s_n^{l}s_{n-1}^{--}$ is a strict prefix of $L_m$, and $L_m$ is a prefix of the central word $s_n^{l+1}s_{n-1}^{--}$. The word $z=s_{n+1}s_n^{--}=s_n^{a_{n+1}+1}s_{n-1}^{--}$ is central, and so we have
\begin{center}
	$r(z,m)=q_n=r(T^{q_n}(z),m)=\cdots=r(T^{(a_{n+1}+1-l)q_n}(c_\alpha),m),$
\end{center}
and this shows that $c_{\alpha}$ turns at least $a_{n+1}-l$ times around the referent cycle.

Since $s_{n+1}s_n$ is a prefix of $c_\alpha$, the word $s_{n+1}s_n=zt_n=s_n^{a_{n+1}+1}s_{n-1}^{--}t_n$ is a prefix of $c_\alpha$ and $s_ns_{n-1}^{--}t_n$ is a prefix of $T^{a_{n+1}q_n}(c_\alpha)$, and from there we see that the word $T^{a_{n+1}q_n}(c_\alpha)$ first goes through the arrow $R_m\rightarrow R_m^*t_{n}^{-}$, which does not belong to the referent cycle. This shows that $c_\alpha$ does not turn $a_{n+1}+1-l$ times around the referent cycle. \end{proof}

We end by a short lemma that states that the non-referent cycle is characterized as the cycle for whom no sturmian words of slope $\alpha$ turns twice around.

\begin{lem}\label{tournenonreferent}
Let $x$ be a sturmian word of slope $\alpha$ and $m>0$. Then, in $G_m$, $x$ does not turn twice around the non-referent cycle. The non-referent cycle is characterized as the cycle such that no sturmian word of slope $\alpha$ turns twice around it.
\end{lem}

\begin{proof}
Since the set of factors of a sturmian word is stable by reversal, we see that if $s\rightarrow t$ is an arrow of $G_m$, then $\tilde{t}\rightarrow \tilde{s}$ is also an arrow of $G_m$. Since the two cycles of $G_m$ have different lengths, and since only one of the two arrows $R_m\rightarrow R_m^*t_{n-1}^{-}$ and $R_m\rightarrow R_m^*t_{n}^{-}$ belongs to the referent cycle, we deduce that only one of the two arrows $0L_m^{-}\rightarrow L_m$ and $1L_m^{-}\rightarrow L_m$ belongs to the referent cycle. The two words $0c_\alpha$ and $1c_\alpha$ are sturmian, and hence there exists a unique finite word $u$ of length $q_n$ such that $uc_\alpha$ turns around the referent cycle. Since $c_\alpha$ always turns at least once around the referent cycle (a property that could characterise the referent cycle), the word $uc_\alpha$ turns at least twice around the referent cycle. Since $x$ and $c_\alpha$ share the same set of factors, we can see that there exists a suffix of $x$ that turns twice around the referent cycle.

If there is a sturmian word $x$ that turns twice around the non-referent cycle, we see that the central word $w$ defined by the common part of the two cycles of $G_m$ is such that the four words $0w0$, $1w0$, $0w1$ and $1w1$ are factors of $x$. But this contradicts the balanced property of sturmian words.

\end{proof}

\section{$\alpha$-numbers and Formal intercepts of sturmian words}

\subsection{Formal intercepts of sturmian words}

The notion of intercepts of sturmian words is deeply involved in study of arithmetic sequences $(n\alpha+\rho \text{\texttt{\emph{[mod}} }1\text{\texttt{\emph{]}}})_{n\geq 0}$, along with numeration systems, see \cite{ ito1, justin, lesca, peltomaki2, ramshaw}. Sturmian words can be obtained as so-called "mechanical words" : with $\alpha$ and $\rho$ two real numbers with $0\leq \alpha \leq 1$, the upper mechanical words $\overline{s}_{\alpha, \rho}: \mathbb{N} \rightarrow \{0,1\}$ and lower mechanical words $\underline{s}_{\alpha, \rho}: \mathbb{N} \rightarrow \{0,1\}$ are defined, for $n\geq 0$ by :
\begin{center}
	$\overline{s}_{\alpha, \rho}(n)=\left\lfloor (n+1)\alpha+\rho \right\rfloor- \left\lfloor n\alpha+\rho \right\rfloor$,
	
		$\underline{s}_{\alpha, \rho}(n)=\left\lceil (n+1)\alpha+\rho \right\rceil- \left\lceil n\alpha+\rho \right\rceil$.
\end{center}
Sturmian words can also be obtained as coding of rotations. For $\rho$ a point on the circle, we consider the ergodic dynamical system associated to the rotation $R_\alpha$ with angle $\alpha$ :
\begin{center}
	$R_\alpha : x\in \mathbb{R}/\mathbb{Z} \longmapsto x+\alpha \in \mathbb{R}/\mathbb{Z}$,
\end{center}
with $\rho$ taken as a starting point of this dynamical system. With the partition given by the intervals $[0,1- \alpha [$ and $[1-\alpha, 1[$ tranferred on the circle, we get the equivalence $\underline{s}_{\alpha,\rho}(n)=0 \Leftrightarrow R_\alpha^n(\rho)\in [0,1- \alpha [$, for $n\geq 0$.

More precisely, for $\alpha=[0,a_1,a_2,a_3,\cdots]$ a slope with continuants $(q_n)_{n\geq -1}$, every real number $x$ satisfying $-\alpha\leq x \leq 1-\alpha$ writes uniquely in the form :
\begin{center}
	$\displaystyle x = \sum_{i\geq 0} b_{i+1}(p_i-\alpha q_i)$
\end{center}
where the coefficients $(b_i)_{i\geq 1}$ satisfy the Ostrowski condition (see \cite{berthe1, ramshaw}), that we study in the rest of the paper. This numeration system is called the Ostrowski numeration system, and is at the center of multiple work in different settings, see \cite{descombes2, hardylittlewood2, ito3, sidorov}. The particular case of the Fibonacci sequence, given by the slope $1/\varphi$ where $\varphi$ is the golden ratio, has been deeply studies as the Zeckendorf numeration system, see \cite{shallitauto, best, kilic, klein, zeckendorf2}.

For $\alpha=[0,a_1,a_2,a_3,\cdots]$ a slope with continuants $(q_n)_{n\geq -1}$, let $ N=\sum_{i=0}^{k-1}b_{i+1}q_i$ with $b_i\geq 0$ for all $i\geq 1$, and let $n\geq 1$. The following assertions are equivalents :
\begin{enumerate}[$i)$]
	\item $\forall l=1\ldots k$, \quad $\displaystyle\sum_{i=0}^{l-1}b_{i+1}q_i<q_{l}$,
	\item we have :	\begin{itemize}	\item $0\leq b_1 \leq a_1-1$,
		\item $\forall i\geq 1$, $0\leq b_i \leq a_i$,
		\item $\forall i\geq 1$, $b_{i+1}=a_{i+1}\Rightarrow b_i=0$.
	\end{itemize}
\end{enumerate}

The conditions in $ii)$ above on the sequence $(b_i)_{i\geq 0}$ are called the Ostrowski conditions. They garantee the unicity in the Ostrowski numeration system : every integer $N\in [0,q_n[$ write uniquely in the form
	\begin{center}
		$\displaystyle N=\sum_{i=0}^{n-1}b_{i+1}q_i$
	\end{center}
where $(b_i)_{i\geq 1}$ satisfies to the Ostrowski conditions. See the very beautiful and complete reference \cite{shallitauto} (theorem 3.9.1) for more details and proofs on these statements.

\begin{defi}\label{interceptdef}
For $\alpha=[0,a_1,a_2,\cdots]$ a slope with continuants $(q_n)_{n\geq-1}$, we define the set of $\alpha$-numbers as the set :
\begin{center}
	$\displaystyle\mathcal{I}_\alpha=\left\{\left.(k_n)_{n>0}\in\prod_{n>0}[0,q_{n}[ \ \right| \ \forall n\geq 0,\ k_n=k_{n+1} \text{\texttt{\emph{[mod}} }q_n\text{\texttt{\emph{]}}} \right\}$.
\end{center}
\end{defi}

If $\rho=(\rho_n)_{n\geq 0}$ is an $\alpha$-number, there exists a unique sequence of integers $(b_i)_{i\geq 1}$, satisfying the Ostrowski conditions, such that :
\begin{center}
	$\displaystyle \rho_n=\sum_{i=0}^{n-1}b_{i+1}q_i$
\end{center}
for all $n\geq 0$. In this case we write directly :
\begin{center}
	$\displaystyle \rho=\sum_{i=0}^{+\infty}b_{i+1}q_i$.
\end{center}

For $n>0$, we set :
\begin{center}
	$\Psi_n^{n+1}$ : $\begin{matrix}
		 [0,q_{n+1}[ & \longmapsto & [0,q_{n}[  \\
		 k & \longmapsto & k $\texttt{ [mod $q_n$]}$  \\
		\end{matrix}$,
\end{center}
and for integers $m\geq n >0$ :
\begin{center}
	$\Psi_n^{m}=\Psi_n^{n+1}\circ \Psi_{n+1}^{n+2}\circ \cdots \circ\Psi_{m-1}^{m}\ :\ [0,q_{m}[ \ \longrightarrow [0,q_{n}[$,
\end{center}
then
\begin{center}
	$\displaystyle \mathcal{I}_\alpha=\lim_{\longleftarrow}[0,q_n[=\left\{\left. (k_n)_{n>0}\in\prod_{n>0}[0,q_{n}[ \ \right| \ n\leq m \Longrightarrow \Psi_n^{m}(k_m)=k_n \right\}$
\end{center}
is the projective limit of the sets $[0,q_n[$ endowed with the functions $\Psi_n^{m}$. The projective limit gives rise to naturally defined functions $\Psi_m : \mathcal{I}_\alpha \rightarrow [0,q_m[$ for $m\geq 1$, defined by
\begin{center}
	$\displaystyle \Psi_m\left(\sum_{i\geq 0}b_{i+1}q_i\right)=\sum_{i=0}^{m-1}b_{i+1}q_i$.
\end{center}
where the coefficients $(b_i)_{i\geq 1}$ satisfy the Ostrowski conditions. This construction of $\alpha$-numbers is to be compared with the construction of $p$-adic numbers, and in the reference \cite{rama} can be found a similar construction for $p$-adic numbers.

\begin{prop}\label{lambdaprefixe}
Let $\rho=\sum_{i\geq 0}b_{i+1}q_i$ be an $\alpha$-number associated to a slope $\alpha$ with continuants $(q_n)_{n\geq-1}$, and $n\geq 1$. Let	
\begin{center}
	$\lambda_n=q_{n+1}+q_n-\rho_{n+1}-2$,
\end{center}
then
\begin{enumerate}
	\item the words $T^{\rho_{n}}(c_\alpha)$ and $T^{\rho_{n+1}}(c_\alpha)$ share the same prefixes of length $\lambda_n$,
\item if $b_{n+1}\neq 0$, then $\lambda_n$ is the length of the longest common prefix between $T^{\rho_{n}}(c_\alpha)$ and $T^{\rho_{n+1}}(c_\alpha)$,
\item the increasing sequence $(\lambda_n)_{n\geq 1}$ tends towards infinity with $n$.
\end{enumerate}
\end{prop}
	
\begin{proof}
1) Let $m=q_n-1\in I_n^0$. We saw that the word $T^{b_{n+1}q_n}(c_\alpha)$ turns $a_{n+1}-b_{n+1}$ times around the referent cycle, before turning around the non-referent cycle. This shows that the two words $T^{b_{n+1}q_n}(c_\alpha)$ and $c_\alpha$ share the same prefix of length $m+(a_{n+1}-b_{n+1})q_n+r$, where $r$ is the length of the common part of the two cycles that constitute $G_m$. Since $m=q_n-1$, every vertex of the non-referent cycle is on the referent cycle. This implies that $r=q_{n-1}-1$, and the two words $T^{b_{n+1}q_n}(c_\alpha)$ and $c_\alpha$ share the same prefix of length $m+(a_{n+1}-b_{n+1})q_n+r=q_n-1+(a_{n+1}-b_{n+1})q_n+q_{n-1}-1$. We deduce that the two words 
\begin{center}
	$T^{\rho_n}(T^{b_{n+1}q_n}(c_\alpha))=T^{\rho_{n+1}}(c_\alpha)$ \quad and \quad $T^{\rho_n}(c_\alpha)$
\end{center}
share the same prefix of length
\begin{center}
$q_n+(a_{n+1}-b_{n+1})q_n+q_{n-1}-2-\rho_n =q_{n+1}+q_n-\rho_{n+1}-2 =\lambda_n$,
\end{center}
where the result comes from.

2) If $b_{n+1}\neq 0$, then the longest common prefix of the words $T^{b_{n+1}q_n}(c_\alpha)$ and $c_\alpha$ is of length $q_n-1+(a_{n+1}-b_{n+1})q_n+q_{n-1}-1$, and from there we get that the length of the longest common prefix of $T^{\rho_{n+1}}(c_\alpha)$ and $T^{\rho_n}(c_\alpha)$ equals $\lambda_n$.

3) We have
\begin{align*}
\lambda_{n+1}-\lambda_n&=q_{n+2}+q_{n+1}-q_{n+1}-q_n-(\rho_{n+2}-\rho_{n+1})\\&=(a_{n+2}-b_{n+2})q_{n+1}  \geq 0,
\end{align*}
from where we get that the sequence $(\lambda_n)_{n\geq 0}$ is increasing. Since $\rho_{n+1}<q_{n+1}$, we get $\lambda_n\geq q_n-1$ and this implies that $\lambda_n \rightarrow +\infty$ when $n\rightarrow +\infty$. \end{proof}

\begin{defi}\label{deftrho}
Let $\rho$ be an $\alpha$-number of the slope $\alpha$. We define the sturmian word $T^{\rho}(c_\alpha)$ of slope $\alpha$ of formal intercept $\rho$ as the word
\begin{center}
	$T^{\rho}(c_\alpha)=\lim T^{\rho_n}(c_\alpha)$
\end{center}
having the same prefix of length $q_n-1$ as $T^{\rho_n}(c_\alpha)$ for all $n\geq 1$.
\end{defi}

\begin{prop}\label{prefixecommun}
Let $\rho$ be an $\alpha$-number of the slope $\alpha$ and $n\geq 1$. Then the length of the longest common prefix of the words
\begin{center}
	$T^{\rho}(c_\alpha)$ \quad and \quad $T^{\rho_n}(c_\alpha)$
\end{center}
equals $\lambda_N$, where $N$ is the smallest integer $N\geq n$ such that $b_{N+1}\neq 0$, where the sequence $(\lambda_n)_{n\geq 1}$ is defined in proposition \ref{lambdaprefixe} above. If no such $N$ exists, the two words are equal.
\end{prop}

	\begin{proof}
	This is a consequence of the equality $\rho_n=\rho_k$ for every $n\leq k \leq N$ when $N$ exists. The second part is a consequence of the equality $\rho_n=\rho_k$ for all $n\leq k$, hence the result.
	\end{proof}

\begin{theo}\label{interceptsturmien}
Let $x$ be a sturmian word of slope $\alpha$. Then there exists a unique $\alpha$-number of the slope $\alpha$ such that
\begin{center}
	 $x=T^{\rho}(c_\alpha)$.
\end{center}
This $\alpha$-number is designated as the formal intercept of $x$.
\end{theo}

\begin{proof}
We consider the sequence defined for $n\geq 0$ by :
\begin{center}
	$\rho_n=\min\{k\geq 0 \ | \ x \text{ and } T^k(c_\alpha) \text{ share the same prefix of length } q_n-1\}$
\end{center}
and we show that the sequence $\rho=(\rho_n)_{n\geq 1}$ defines an $\alpha$-number. Let $n\geq 1$ and $m=q_n-1\in I_n^0$. Since the referent cycle is of length $m=q_n-1$, all the vertexes of $G_m$ are on the referent cycle, and this implies $0\leq\rho_n<q_n$. Since $T^{\rho_{n+1}}(c_\alpha)$, $T^{\rho_{n}}(c_\alpha)$ and $x$ share the same prefix of length $q_n-1$, the paths these words define in $G_m$ start at the same vertex.

We write $\rho_{n+1}=bq_n+c$ with $c<q_n$. Since $\rho_{n+1}=bq_n+c<q_{n+1}=a_{n+1}q_n+q_{n-1}$, we have $b\leq a_{n+1}$ and if $b= a_{n+1}$ then $c< q_{n-1}$. Because the characteristic word $c_\alpha$ turns $a_{n+1}$ times around the referent cycle, if $b<a_{n+1}$ then $T^{\rho_{n+1}}(c_\alpha)$ and $T^{c}(c_\alpha)$ start at the same vertex, and so share the same prefix of length $q_n-1$. Since the referent cycle is of length $q_n$ and $\rho_n<q_n$, we must have $c=\rho_n$. In the case where $b=a_{n+1}$, then $c<q_{n-1}$, so that $T^{\rho_{n+1}}(c_\alpha)$ starts on the common part of the cycles that constitute $G_m$, and we see that $T^{\rho_{n+1}}(c_\alpha)$ and $T^{c}(c_\alpha)$ start on the same vertex, which is on the referent cycle, showing that $\rho_n=c$. Hence $\rho_n=\rho_{n+1} $\texttt{ [mod $q_n$]}, and $\rho=(\rho_n)_{n\geq 1}$ defines an $\alpha$-number, and the word $x$ is easily found to have formal intercept $\rho$.

For unicity, notice that since $m=q_n-1$ and that the referent cycle is of length $q_n$, there can be only one $k<q_n$ such that $T^{k}(c_\alpha)$ and $T^{\rho}(c_\alpha)$ share the same prefix of length $q_n-1$, and since $\rho_n$ is a such $k$, every $\alpha$-number $\gamma$ such that $x=T^{\gamma}(c_\alpha)$ must satisfy $\gamma_n=\rho_n$. This shows unicity, and the theorem.

\end{proof}

\begin{prop}\label{interceptmoinsun}
The two sturmian words $0c_\alpha$ and $1c_\alpha$ are the sturmian words respectively associated to the formal intercepts :
\begin{center}
	$\displaystyle \sum_{i\geq 0} a_{2i+2}q_{2i+1}=(q_{2\left\lfloor n/2\right\rfloor}-1)_{n\geq 1}$ \quad and \quad $\displaystyle (a_1-1)+\sum_{i\geq 1} a_{2i+1}q_{2i}=(q_{2\left\lfloor n/2\right\rfloor +1}-1)_{n\geq 1}$,
\end{center}
and we denote them by $\sigma_0$ and $\sigma_1$.
\end{prop}

\begin{proof}
We check first the two equalities of sequences. We clearly have $q_2-1=q_2-q_0=a_2q_1$, and the result for the first sequence is a consequence of the relation $q_{2(n+1)}-q_{2n}=a_{2n+2}q_{2n+1}$ and an easy induction, the proof for the second sequence being similar.

Let $x$ be the sturmian word of formal intercept $\sum_{i\geq 0} a_{2i+2}q_{2i+1}$, and let us compute $T(x)$. We know that $x$ and $T^{q_{2\left\lfloor n/2\right\rfloor}-1}(c_\alpha)$ share the same prefix of length $q_{2\left\lfloor n/2\right\rfloor}-1$, so that $T(x)$ and $T^{q_{2\left\lfloor n/2\right\rfloor}}(c_\alpha)$ share the same prefix of length $q_{2\left\lfloor n/2\right\rfloor}-2$, given that this quantity tends towards infinity with $n$. But $T^{q_{2\left\lfloor n/2\right\rfloor}}(c_\alpha)$ has the same prefix of length $q_{2\left\lfloor n/2\right\rfloor}-1$ as $c_\alpha$. We deduce from this that $T(x)=c_\alpha$. A similar proof goes for the second $\alpha$-number.

As a consequence of the unicity of the Ostrowski expansion of $\alpha$-number, the two associated sturmian words are distincts. We conclude by noticing that the word $T^{a_1-1}(c_\alpha)$ starts with the letter 1.
\end{proof}

In particular, if $c_\alpha$ is a strict suffix of $T^{\rho}(c_\alpha)$, then there exist $k\geq 0$ and $N\geq 0$ such that, either $\Psi_n(\rho)=q_{2\left\lfloor \frac{n}{2}\right\rfloor}-1-k$ for all $n\geq N$, either $\Psi_n(\rho)=q_{2\left\lfloor \frac{n}{2}\right\rfloor+1}-1-k$ for all $n\geq N$.

In the case of the Fibonacci sequence $(F_n)_{n\geq -1}$ defined by $F_{-1}=0$, $F_0=1$ and $F_{n+1}=F_n+F_{n-1}$ for $n\geq 0$, our formalism and the above result allow us to write the symbolic formulas :
\begin{center}
	$\displaystyle \sum_{n\geq 1} F_{2n}= -1$ \quad and \quad $\displaystyle \sum_{n\geq 0} F_{2n+1}= -1$
\end{center}
and is to be compared to the formula :
	\begin{center}
	$\displaystyle \sum_{i\geq 0} (p-1)p^i=-1$
\end{center}
for $p$-adic numbers. It is a remarkable and deep fact of the Ostrowski numeration system that in the case of $\alpha$-numbers, there are two distincts entities that can bear the symbol "-1", as opposed to only one for $p$-adic numbers.

\subsection{Repetition function, diophantine exponent and irationality measure}

In this subsection we give the values of the repetition function for sturmian words. The values of this function will allow us to obtain a general formula for the diophantine exponent of sturmian numbers, given in the remainder of this section. The works of Y. Bugeaud and D. Kim, or S. Moothathu, give asymptotical results on this function, however determining precise values for such complexity functions is an important topic in symbolic dynamics, see \cite{fici2, justin}.

\begin{prop}\label{valeurreppremier}
Let $x=T^{\rho}(c_\alpha)$ be a sturmian word of slope $\alpha$, where $\alpha$ has continuants $(q_n)_{n\geq -1}$, and $\rho=\sum_{i\geq 0}b_{i+1}q_i$ is an $\alpha$-number of the slope $\alpha$, and let $m\in I_n^l$ for $n\geq 0$ and $0\leq l \leq a_{n+1}-1$. We write $m=(l+1)q_n+q_{n-1}-2-r$ where $r$ is the length of the common part of the two cycles of the graph $G_m$. Then :
\begin{enumerate}
	\item $r(T^{\rho_{n+1}}(c_\alpha),m)=q_n$ if $0\leq\rho_{n+1}\leq (a_{n+1}-l-1)q_n+r$,
	\item $r(T^{\rho_{n+1}}(c_\alpha),m)=q_{n+1}-\rho_{n+1}$ if $(a_{n+1}-l-1)q_n+r<\rho_{n+1}< (a_{n+1}-l)q_n$,
	\item $r(T^{\rho_{n+1}}(c_\alpha),m)=lq_n+q_{n-1}$ if $(a_{n+1}-l)q_n\leq\rho_{n+1}\leq (a_{n+1}-l)q_n+r$,
	\item $r(T^{\rho_{n+1}}(c_\alpha),m)=q_{n+1}-\rho_{n+1}+q_n$ if $(a_{n+1}-l)q_n+r< \rho_{n+1}\leq q_{n+1}-1$.
\end{enumerate}
\end{prop}

\begin{proof}
1) Since the word $c_\alpha$ turns $a_{n+1}-l$ times around the referent cycle, if $0\leq\rho_{n+1}\leq (a_{n+1}-l-1)q_n+r$, then $T^{\rho_{n+1}}(c_\alpha)$ starts on a vertex belonging to the referent cycle and turns around the referent cycle. This implies that
\begin{center}
	$r(T^{\rho_{n+1}}(c_\alpha),m)=q_n$.
\end{center}

2) If $(a_{n+1}-l-1)q_n+r<\rho_{n+1}< (a_{n+1}-l)q_n$, then $T^{\rho_{n+1}}(c_\alpha)$ starts on the referent cycle, but does not turn around it. Since $T^{(a_{n+1}-l)q_n-\rho_{n+1}}(T^{\rho_{n+1}}(c_\alpha))$ starts on the vertex corresponding to the left special factor $L_m$, and that this word does not turn around the referent cycle, it must turn around the non-referent cycle, and we have :
\begin{center}
	$r(T^{\rho_{n+1}}(c_\alpha),m)=(a_{n+1}-l)q_n-\rho_{n+1}+lq_n+q_{n-1}=q_{n+1}-\rho_{n+1}$.
\end{center}

3) If $(a_{n+1}-l)q_n\leq\rho_{n+1}\leq (a_{n+1}-l)q_n+r$, then $T^{\rho_{n+1}}(c_\alpha)$ starts on the common part of the two cycles of $G_m$ of $x$. Since this word does not turn around the referent cycle, it turns around the non-referent cycle, and we have :
\begin{center}
	$r(T^{\rho_{n+1}}(c_\alpha),m)=lq_n+q_{n-1}$.
\end{center}

4) Notice that the prefix of length $m$ of $T^{q_{n+1}-1}(c_\alpha)$ is the only vertex $w$ of $G_m$ such that the arrow $w\rightarrow L_m$ does not belong to the referent cycle. This shows that if $(a_{n+1}-l)q_n+r< \rho_{n+1}\leq q_{n+1}-1$, then $T^{\rho_{n+1}}(c_\alpha)$ starts on a vertex that is on the non-referent cycle and that is not on the referent cycle. Since a sturmian word cannot turn twice around the non-referent cycle, the word $T^{q_{n+1}-\rho_{n+1}}(T^{\rho_{n+1}}(c_\alpha))$ starts at the point $L_m$ of $G_m$ and turns around the referent cycle. This shows that :
\begin{center}
		$r(T^{\rho_{n+1}}(c_\alpha),m)=q_{n+1}-\rho_{n+1}+q_n$.
\end{center}

\end{proof}

\begin{coro}
Let $x=T^{\rho}(c_\alpha)$ be a sturmian word of slope $\alpha$, where $\alpha$ has continuants $(q_n)_{n\geq -1}$, $\rho=\sum_{i\geq 0}b_{i+1}q_i$ is an $\alpha$-number of the slope $\alpha$, and let $m\in I_n=[q_n-1, q_{n+1}-1[$ for $n\geq 0$. Then :
\begin{enumerate}
	\item we have $r(T^{\rho_{n+1}}(c_\alpha),m)=q_n$ if $q_n-1\leq m \leq q_{n+1}-\rho_{n+1}-2$, this case does not appear if $b_{n+1}=a_{n+1}$, or if $b_{n+1}=a_{n+1}-1$ and $b_n\neq 0$,
	\item we have $r(T^{\rho_{n+1}}(c_\alpha),m)=q_{n+1}-\rho_{n+1}$ if $q_{n+1}-\rho_{n+1}-1\leq m\leq q_{n+1}-b_{n+1}q_n-2$, this case does not appear if $b_{n+1}=a_{n+1}$,
	\item if $m\leq q_{n+1}+q_n-\rho_{n+1}-2$ : \begin{itemize}
		\item we have $r(T^{\rho_{n+1}}(c_\alpha),m)=q_{n-1}$ if $b_{n+1}=a_{n+1}$, for $q_n-1\leq m$,
		\item we have $r(T^{\rho_{n+1}}(c_\alpha),m)=q_{n+1}-b_{n+1}q_n$ if $0<b_{n+1}<a_{n+1}$, for $q_{n+1}-b_{n+1}q_n-1\leq m$,
	\end{itemize}
these cases do not appear if $b_{n+1}=0$,
	
	\item we have $r(T^{\rho_{n+1}}(c_\alpha),m)=q_{n+1}-\rho_{n+1}+q_n$ if $q_{n+1}-\rho_{n+1}+q_n-1\leq m \leq q_{n+1}-2$, this case does not appear if $b_{n+1}=0$.
\end{enumerate}
\end{coro}

\begin{proof} We keep the notations of \ref{valeurreppremier}.

1) According to \ref{valeurreppremier}.1),  $r(T^{\rho_{n+1}}(c_\alpha),m)=q_n$ if $0\leq\rho_{n+1}\leq (a_{n+1}-l-1)q_n+r$, which with the relation $m=(l+1)q_n+q_{n-1}-2-r$ provides the inequality :
\begin{center}
	$0\leq \rho_{n+1} \leq q_{n+1}-m-2$
\end{center}
which leads, keeping in mind $m\in I_n$, to :
\begin{center}
	$q_n-1\leq m \leq q_{n+1}-\rho_{n+1}-2$.
\end{center}
This situation cannot appear if $q_{n+1}-\rho_{n+1}-2<q_n-1$, which happens exactly when $b_{n+1}=a_{n+1}$, or when $b_{n+1}=a_{n+1}-1$ and $b_n\neq 0$.

2) According to \ref{valeurreppremier}.2), $r(T^{\rho_{n+1}}(c_\alpha),m)=q_{n+1}-\rho_{n+1}$ if $(a_{n+1}-l-1)q_n+r<\rho_{n+1}< (a_{n+1}-l)q_n$, which implies $b_{n+1}=a_{n+1}-l-1$, and since $l\geq 0$, this case does not happen if $b_{n+1}=a_{n+1}$. The inequality
\begin{center}
	$(a_{n+1}-l-1)q_n+r<\rho_{n+1}< (a_{n+1}-l)q_n$
\end{center}
leads to :
\begin{center}
	$b_{n+1}q_n+r<\rho_{n+1}<(b_{n+1}+1)q_n$
\end{center}
for which the inequality on the right is trivial. The inequality on the left provides, with the help of the relation $m=(l+1)q_n+q_{n-1}-2-r$,
\begin{center}
	$q_{n+1}-2-\rho_{n+1}< m \leq \max I_n^l=(a_{n+1}-b_{n+1})q_n+q_{n-1}-1=q_{n+1}-b_{n+1}q_n$
\end{center}
and this shows 2).

3) According to \ref{valeurreppremier}.3), $r(T^{\rho_{n+1}}(c_\alpha),m)=lq_n+q_{n-1}$ if $(a_{n+1}-l)q_n\leq\rho_{n+1}\leq (a_{n+1}-l)q_n+r$. This implies that $b_{n+1}=a_{n+1}-l$, which is to say that $l=a_{n+1}-b_{n+1}$. Since $0\leq l \leq a_{n+1}-1$, this case cannot happen if $b_{n+1}=0$. On the other hand, the inequality
\begin{center}
	$(a_{n+1}-l)q_n\leq\rho_{n+1}\leq (a_{n+1}-l)q_n+r$
\end{center}
leads, with the help of the relation $m=(l+1)q_n+q_{n-1}-2-r$, to
\begin{center}
	$b_{n+1}q_n\leq\rho_{n+1} \leq q_{n+1}+q_n-2-m$
\end{center}
the inequality on the left being trivial, and this provides $m\leq q_{n+1}+q_n-\rho_{n+1}-2$. The different cases come from the fact that $m\in I_n^{a_{n+1}-b_{n+1}}$.

4) According to \ref{valeurreppremier}.4), $r(T^{\rho_{n+1}}(c_\alpha),m)=q_{n+1}-\rho_{n+1}+q_n$ if 
\begin{center}
	$(a_{n+1}-l)q_n+r< \rho_{n+1}\leq q_{n+1}-1$,
\end{center}
the inequality on the right being trivial. The inequality on the left leads to
\begin{center}
	$q_{n+1}+q_n-\rho_{n+1}-2 \leq m$
\end{center}
which is possible, given $m\in I_n$, only if $q_n\leq \rho_{n+1}$, which is equivalent to $b_{n+1}\neq 0$.
\end{proof}

\begin{prop}
Let $x=T^{\rho}(c_\alpha)$ be a sturmian word of slope $\alpha$, where $\alpha$ has continuants $(q_n)_{n\geq -1}$, $\rho=\sum_{i\geq 0}b_{i+1}q_i$ is an $\alpha$-number of the slope $\alpha$, and let $m\in I_n^l$ for $n\geq 0$ and $0\leq l \leq a_{n+1}-1$. We write $m=(l+1)q_n+q_{n-1}-2-r$ where $r$ is the length of the common part of the two cycles of $G_m$.

If $b_{n+2}\neq a_{n+2}$, then $r(x,m)=r(T^{\rho_{n+1}}(c_\alpha),m)$. If $b_{n+2}=a_{n+2}$, then $r(x,m)=r(T^{\rho_{n+2}}(c_\alpha),m)$.
\end{prop}

\begin{proof}
For the first part, we show that the quantity $r(T^{\rho_{n+1}}(c_\alpha),m)+m$ is smaller than the length of the longest common prefix of the two words $T^{\rho_{n+1}}(c_\alpha)$ and $T^{\rho}(c_\alpha)$. The quantity $r(x,m)+m$ is characterized as the length of the smallest prefix $w$ of $x$ having a factor of length $m$ appearing twice in $w$. Let $N$ be the smallest $N\geq n+1$ such that $b_{N+1}\neq 0$, so that the longest common prefix of the words $T^{\rho_{n+1}}(c_\alpha)$ and $T^{\rho}(c_\alpha)$ has length $\lambda_N=q_{N+1}+q_N-2-\rho_{N+1}$. We proceed case by case :

1) If $0\leq\rho_{n+1}\leq (a_{n+1}-l-1)q_n+r$, then :
\begin{align*}
r(T^{\rho_{n+1}}(c_\alpha),m)+m & = q_n+(l+1)q_n+q_{n-1}-2-r \\
 & \leq q_n +q_{n-1} + a_{n+1}q_n-\rho_{n+1}-2 \\
 & = q_{n+1}+q_n-\rho_{n+1}-2 \\
 & =\lambda_n \leq \lambda_{N}
\end{align*}

2) If $(a_{n+1}-l-1)q_n+r<\rho_{n+1}< (a_{n+1}-l)q_n$, then we must have $b_{n+1}=a_{n+1}-l-1$, and so :
\begin{align*}
r(T^{\rho_{n+1}}(c_\alpha),m)+m & = q_{n+1}-\rho_{n+1}+(l+1)q_n+q_{n-1}-2-r \\
 & = q_{n+1}-\rho_{n+1}+(a_{n+1}-b_{n+1})q_n+q_{n-1}-2-r
\end{align*}

We proceed by equivalences :
\begin{align*}
&  & r(T^{\rho_{n+1}}(c_\alpha),m)+m+r & \leq \lambda_{n+1} \\
& \Longleftrightarrow & q_{n+1}-\rho_{n+1}+(a_{n+1}-b_{n+1})q_n+q_{n-1} & \leq (a_{n+2}-b_{n+2})q_{n+1}+q_n+q_{n+1}-\rho_{n+1} \\
& \Longleftrightarrow & (a_{n+1}-b_{n+1})q_n+q_{n-1} & \leq (a_{n+2}-b_{n+2})q_{n+1}+q_n,
\end{align*}
which is true if $b_{n+2}< a_{n+2}$. 

3) If $(a_{n+1}-l)q_n\leq\rho_{n+1}\leq (a_{n+1}-l)q_n+r$, then we must have $b_{n+1}=a_{n+1}-l$ since $r<q_n$. Added with $0\leq l \leq a_{n+1}-1$, we have $b_{n+1}\neq 0$ and so we are in the case where $b_{n+2}\neq a_{n+2}$. Then :
\begin{align*}
r(T^{\rho_{n+1}}(c_\alpha),m)+m & = lq_n+q_{n-1}+(l+1)q_n+q_{n-1}-2-r \\
 & = 2(a_{n+1}-b_{n+1})q_n+2q_{n-1}+q_n-2-r
\end{align*}
We proceed by equivalences :
\begin{align*}
&  & r(T^{\rho_{n+1}}(c_\alpha),m)+m +r & \leq \lambda_{n+1} \\
& \Longleftrightarrow & 2(a_{n+1}-b_{n+1})q_n+2q_{n-1}+q_n & \leq (a_{n+2}-b_{n+2})q_{n+1}+q_n+q_{n+1}-\rho_{n+1} \\
& \Longleftrightarrow & 2(a_{n+1}-b_{n+1})q_n+2q_{n-1} & \leq (a_{n+2}-b_{n+2})q_{n+1}+(a_{n+1}-b_{n+1})q_n+q_{n-1}-\rho_n \\
& \Longleftrightarrow & (a_{n+1}-b_{n+1})q_n+q_{n-1}+\rho_n & \leq (a_{n+2}-b_{n+2})q_{n+1}.
\end{align*}
Since $b_{n+1}\neq 0$ and $b_{n+2}\neq a_{n+2}$, $(a_{n+1}-b_{n+1})q_n\leq (a_{n+1}-1)q_n$ and $q_{n+1}\leq (a_{n+2}-b_{n+2})q_{n+1}$, so that :
\begin{align*}
(a_{n+1}-b_{n+1})q_n+q_{n-1}+\rho_n & \leq (a_{n+1}-1)q_n+q_{n-1}+\rho_n
  = q_{n+1}+\rho_n-q_n \\
 & \leq q_{n+1} \\
 & \leq (a_{n+2}-b_{n+2})q_{n+1}
\end{align*}
showing the result.

4) If $(a_{n+1}-l)q_n+r< \rho_{n+1}\leq q_{n+1}-1$, then $0<(a_{n+1}-l)\leq b_{n+1}$ so that $b_{n+2}\neq a_{n+2}$ and we have :
\begin{align*}
&  & r(T^{\rho_{n+1}}(c_\alpha),m)+m+r & \leq \lambda_{n+1} \\
& \Longleftrightarrow & q_{n+1}+q_n-\rho_{n+1}+(l+1)q_n+q_{n-1} & \leq (a_{n+2}-b_{n+2})q_{n+1}+q_n+q_{n+1}-\rho_{n+1} \\
& \Longleftrightarrow & (l+1)q_n+q_{n-1} & \leq (a_{n+2}-b_{n+2})q_{n+1},
\end{align*}
but this is true since $a_{n+2}-b_{n+2}\geq 1$ and $l+1\leq a_{n+1}$.

Notice that the case 2) is the only one where we use our hypothesis $b_{n+2}<a_{n+2}$. Since the length of the longest common prefix of $T^{\rho_{n+1}}(c_\alpha)$ and $T^{\rho}(c_\alpha)$ is smaller than the length of the longest common prefix of $T^{\rho_{n+2}}(c_\alpha)$ and $T^{\rho}(c_\alpha)$, from the fact that the sequence $(\lambda_n)_{n\geq 1}$ increases, the equality $r(T^{\rho}(c_\alpha),m)=r(T^{\rho_{n+1}}(c_\alpha),m)$ and the fact that the quantity $r(T^{\rho_{n+1}}(c_\alpha),m)+m$ is smaller than the length of the longest common prefix of the words $T^{\rho_{n+1}}(c_\alpha)$ and $T^{\rho}(c_\alpha)$ implies the equality $r(T^{\rho}(c_\alpha),m)=r(T^{\rho_{n+2}}(c_\alpha),m)$. This shows the second part of the assertion, excepted the case 2).

Our hypothesis $b_{n+2}=a_{n+2}$ implies $b_{n+1}=0$ and so $\rho_{n+1}=\rho_n$. The inequality of case 2), given by $(a_{n+1}-l-1)q_n+r<\rho_{n+1}< (a_{n+1}-l)q_n$ implies $l=a_{n+1}-1$. The length of the longest common prefix of $T^{\rho_{n+1}}(c_\alpha)=T^{\rho_{n}}(c_\alpha)$ and $T^{\rho}(c_\alpha)$ equals $\lambda_{n+1}=q_{n+2}+q_{n+1}-\rho_{n+2}-2=q_{n+1}+q_n-\rho_n-2=\lambda_n$. Since $m=q_{n+1}-2-r$, we have $\lambda_n=q_n-\rho_n+m+r$, and this implies that the paths taken by the words $T^{\rho}(c_\alpha)$ and $T^{\rho_n}(c_\alpha)$ diverge at the point $R_m$ corresponding to the right special factor. Since, according to theorem \ref{rauzystructure}, the word $T^{\rho_n}(c_\alpha)$ must take the arrow toward the non-referent cycle at this moment, this means that the word $T^{\rho}(c_\alpha)$ must turn around the referent cycle. This argument, applied to the integral $\alpha$-number defined by the number $\rho_{n+2}$, shows in a similar manner that the word $T^{\rho_{n+2}}(c_\alpha)$ turns around the referent cycle. This shows that 
\begin{center}
	$r(T^{\rho}(c_\alpha),m)=r(T^{\rho_{n+2}}(c_\alpha),m)=q_n$
\end{center}
in this case. Hence the second part of the proposition.
\end{proof}

\begin{theo}\label{repetitionvaleur}
Let $x=T^{\rho}(c_\alpha)$ be a sturmian word of slope $\alpha=[0,a_1, a_2,a_3,\ldots]$ with continuants $(q_n)_{n\geq -1}$, where $\rho=\sum_{i\geq 0}b_{i+1}q_i$ is an $\alpha$-number of the slope $\alpha$, and let $m\in I_n=[q_n-1,q_{n+1}-1[$. Then :
\begin{enumerate}
	\item if $b_{n+1}=0$ and $b_{n+2}=a_{n+2}$, then 
	\begin{itemize}
		\item $r(T^{\rho}(c_\alpha),m)=q_{n}$ for $q_n-1\leq m \leq q_{n+1}-2$,
	\end{itemize}

	\item if $b_{n+1}=0$, $b_{n+2}\neq a_{n+2}$ and $b_n=0$, then 
	\begin{itemize}
		\item $r(T^{\rho}(c_\alpha),m)=q_{n}$ for $q_n-1\leq m \leq q_{n+1}-\rho_{n-1}-2$,
		\item $r(T^{\rho}(c_\alpha),m)=q_{n+1}-\rho_{n-1}$ for $q_{n+1}-\rho_{n-1}-1\leq m \leq q_{n+1}-2$,
	\end{itemize}
	
	\item if $b_{n+1}=0$, $b_{n+2}\neq a_{n+2}$ and $a_{n+1}\neq 1$, then 
\begin{itemize}
		\item $r(T^{\rho}(c_\alpha),m)=q_{n}$ for $q_n-1\leq m \leq q_{n+1}-\rho_{n}-2$,
		\item $r(T^{\rho}(c_\alpha),m)=q_{n+1}-\rho_{n}$ for $q_{n+1}-\rho_{n}-1\leq m \leq q_{n+1}-2$,
	\end{itemize}
	
		\item if $b_{n+1}=0$, $b_{n+2}\neq a_{n+2}$, $a_{n+1}= 1$ and $b_n\neq 0$, then 
\begin{itemize}
		\item $r(T^{\rho}(c_\alpha),m)=q_{n+1}-\rho_{n}=q_{n}+q_{n-1}-\rho_{n}$ for $q_n-1\leq m \leq q_{n+1}-2$,
		\end{itemize}
		
		\item If $0<b_{n+1}<a_{n+1}-1$, then 
\begin{itemize}
			\item $r(T^{\rho}(c_\alpha),m)=q_{n}$ for $q_n-1\leq m \leq q_{n+1}-\rho_{n+1}-2$,
			\item $r(T^{\rho}(c_\alpha),m)=q_{n+1}-\rho_{n+1}$ for $q_{n+1}-\rho_{n+1}-1\leq m \leq q_{n+1}-b_{n+1}q_n-2$,
			\item $r(T^{\rho}(c_\alpha),m)=q_{n+1}-b_{n+1}q_n$ for $q_{n+1}-b_{n+1}q_n-1\leq m \leq q_{n+1}+q_{n}-\rho_{n+1}-2$,
			\item $r(T^{\rho}(c_\alpha),m)=q_{n+1}+q_{n}-\rho_{n+1}$ for $q_{n+1}+q_{n}-\rho_{n+1}-1\leq m \leq q_{n+1}-2$,
		\end{itemize}
		
		\item if $b_{n+1}=a_{n+1}-1$, $a_{n+1}\neq 1$ and $b_n=0$, then 
\begin{itemize}
			\item $r(T^{\rho}(c_\alpha),m)=q_{n}$ for $q_n-1\leq m \leq q_{n}+q_{n-1}-\rho_{n-1}-2$,
			\item $r(T^{\rho}(c_\alpha),m)=q_{n}+q_{n-1}-\rho_{n-1}$ for $q_{n}+q_{n-1}-\rho_{n-1}-1\leq m \leq q_n+q_{n-1}-2$,
			\item $r(T^{\rho}(c_\alpha),m)=q_n+q_{n-1}$ for $q_n+q_{n-1}-1\leq m \leq 2q_{n}+q_{n-1}-\rho_{n-1}-2$,
			\item $r(T^{\rho}(c_\alpha),m)=2q_{n}+q_{n-1}-\rho_{n-1}$ for $2q_{n}+q_{n-1}-\rho_{n-1}-1\leq m \leq q_{n+1}-2$,
		\end{itemize}

		\item if $b_{n+1}=a_{n+1}-1$, $a_{n+1}\neq 1$ and $b_n\neq 0$, then 
\begin{itemize}
			\item $r(T^{\rho}(c_\alpha),m)=q_{n}+q_{n-1}-\rho_{n-1}$ for $q_{n}-1\leq m \leq q_n+q_{n-1}-2$,
			\item $r(T^{\rho}(c_\alpha),m)=q_n+q_{n-1}$ for $q_n+q_{n-1}-1\leq m \leq 2q_{n}+q_{n-1}-\rho_{n-1}-2$,
			\item $r(T^{\rho}(c_\alpha),m)=2q_{n}+q_{n-1}-\rho_{n-1}$ for $2q_{n}+q_{n-1}-\rho_{n-1}-1\leq m \leq q_{n+1}-2$,
		\end{itemize}
		
			\item if $b_{n+1}=a_{n+1}$, then 
\begin{itemize}
			\item $r(T^{\rho}(c_\alpha),m)=q_{n-1}$ for $q_{n}-1\leq m \leq q_n+q_{n-1}-\rho_{n-1}-2$,
			\item $r(T^{\rho}(c_\alpha),m)=q_n+q_{n-1}-\rho_{n-1}$ for $q_n+q_{n-1}-\rho_{n-1}-1\leq m \leq q_{n+1}-2$.
		\end{itemize}
	\end{enumerate}
		
\end{theo}

Recall that for a real number $x$, its irrationality exponent $\mu(x)$ is defined as the supremum on real numbers $\mu$ such that the inequality :
\begin{center}
	$\displaystyle \left|x- \frac{p}{q} \right|<\frac{1}{q^\mu}$
\end{center}
has infinitely many solutions in the irreducible fraction $p/q$. The irrationality exponent of a real number is very hard to compute, see \cite{adamczewski2, becher, bugeaud1, bugeaud3, fischler, rhin, roth}, these references include the acclaimed result of K. Roth.


We present now two exponents defined on infinite words. A link between real numbers and infinite words is given by expansion of a real number in an integer base. For $u$ a finite word, and $w\geq 0$ a real number, we define the fractional power $u^{w}$ as the finite word $u^w=u^{\left\lfloor w \right\rfloor}v$ where $v$ is the prefix of $u$ with length $\displaystyle \left\lfloor (w-\left\lfloor w \right\rfloor )|u| \right\rfloor$, so that $u^{k/|u|}$ is the prefix of length $k$ of $u$, for $0\leq k \leq  |u|$. The critical exponent and initial critical exponent have been introduced by L. Zamboni, C. Holton and V. Berthé in \cite{holton}, where they give a general formula for their values on sturmian words.

\begin{defi}
For $x$ an infinite word over an alphabet $\mathcal{A}$, we denote by $ce(x)$ (resp. $ice(x)$) the critical exponent (resp. initial critical exponent) of $x$ as the supremum of real numbers $w\geq 0$ such that there exists infinitely many words $u$ such that $u^w$ is a factor (resp. prefix) of $x$.
\end{defi}

The diophantine exponent, defined below, has been introduced by B. Adamczewski and Y. Bugeaud in \cite{adamczewskibug}. In the other contribution \cite{adamczewski}, they show that for an infinite word $x=x_1x_2x_3\cdots$ over a digital alphabet $\mathcal{A}=\{0,1,\ldots , b-1\}$ for some $b\geq 2$, such that $dio(x)$ is finite and such that $x$ has sublinear complexity, meaning that there exists a constant $C>1$ such that $p(x,n)\leq Cn$ for all $n\geq 1$, then the irrationality exponent of the real number
\begin{center}
	$y= \displaystyle \sum_{k \geq 1}\frac{x_k}{b^k}$
\end{center}
is subjected to the inequalities
\begin{center}
	$dio(x) \leq \mu(y) \leq (2C+1)^3(dio(x)+1)$
\end{center}
and so is also finite.

\begin{defi}
For $x$ an infinite word over an alphabet $\mathcal{A}$, we denote $dio(x)$ the diophantine exponent of $x$, defined as the supremum of real numbers $\rho \geq 0$ such that there exist two sequences $(u_n)_{n\geq 1}$ and $(v_n)_{n\geq 1}$ of finite words and a sequence of real numbers $(w_n)_{n\geq 1}$ satisfying the three following conditions :
\begin{enumerate}
	\item $u_nv_n$ is a prefix of $x$ for all $n\geq 1$,
	\item the sequence of length $(|v_n^{w_n}|)_{n\geq 1}$ is strictly increasing,
	\item $\displaystyle \frac{|u_nv_n^{w_n}|}{|u_nv_n|}\geq \rho$.
\end{enumerate}
\end{defi}

In \cite{adamczewski} is showed that for real numbers of the form :
\begin{center}
	$y=\displaystyle \sum_{k\geq 0}\frac{1}{b^{\left\lfloor k\gamma \right\rfloor}}=0,x_1x_2x_3\cdots$
\end{center}
where $x=x_1x_2x_3\cdots$ encodes the base $b$ expansion of $y$, we have
\begin{center}
	$\mu(y)=dio(x)= 1+\limsup_{n}[a_n,a_{n-1},\ldots, a_1]=dio(c_{1/\gamma})$
\end{center}
where $\gamma >1$ is an irrational number with continued fraction having partial quotients $(a_k)_{k\geq 1}$. B. Adamczewski and Y. Bugeaud show that for a given sturmian sequence, its diophantine exponent is finite if and only if the sequence of the partial quotients of its slope is bounded. These results are based upon the computation :
\begin{center}
	$\displaystyle dio(c_\alpha)=1+\limsup \frac{q_{n+1}}{q_n}$
\end{center}
for a slope $\alpha$ with continuants $(q_n)_{n \geq -1}$.

Also, from the work of Y. Bugeaud and D. Kim \cite{bugeaudkim1}, the diophantine exponent of an infinite word is linked to the repetition function by the formula
\begin{center}
	$\displaystyle (dio(x)-1)\liminf \frac{r(x,n)}{n}= 1$
\end{center}
valid for any infinite word $x$ over a finite alphabet $\mathcal{A}$. In this reference is also showed the equality
\begin{center}
	$\displaystyle dio(x)= \mu\left(\sum_{k\geq 0}\frac{x_k}{b^k}\right)$
\end{center}
for a sturmian word $x=x_0x_1x_2\cdots$.

\begin{theo}\label{diosturm}
Let $x=x_0x_1x_2\cdots=T^{\rho}(c_\alpha)$ be a sturmian word of slope $\alpha$. We write $\rho=\sum_{i\geq 0}b_{i+1}q_i$ where $(b_i)_{i\geq 1}$ satisfy the Ostrowski conditions and such that $0<b_{i}< a_i-1$ for all $i\geq 1$ large enough. Then we have the general formula :
\begin{align*}
	\displaystyle dio(x)&= \mu\left(\sum_{k\geq 0}\frac{x_k}{b^k}\right) \\ &=1+ \limsup\left(\frac{q_{n+1}-\rho_{n+1}}{q_n}, \frac{q_{n+1}-b_{n+1}q_n}{q_{n+1}-\rho_{n+1}}, \frac{q_{n+1}-\rho_{n+1}+q_n}{q_{n+1}-b_{n+1}q_n},\frac{q_{n+1}}{q_{n+1}-\rho_{n+1}+q_n} \right).
\end{align*}
\end{theo}

\begin{proof}
According to \cite{bugeaudkim1}, we have
\begin{center}
	$\displaystyle dio(x)= \mu\left(\sum_{k\geq 0}\frac{x_k}{b^k}\right)=1+\frac{1}{\liminf \frac{r(x,n)}{n}} =1+\limsup \frac{n}{r(x,n)}\cdot$
\end{center}

Let $r(x,\mathbb{N}^*)=\{v_1<v_2<v_3<\cdots\}$ be the set of values of the repetition function associated to $x$. The result \ref{repetitioncomportement}, namely the equivalence $r(x,m-1)\neq r(x,m) \Leftrightarrow r(x,m)=m+1$, added with the increasing of the repetition function, allow us to reduce the limsup above to the limsup of the subsequence formed by the values of the repetition function. We get :
\begin{center}
	$\displaystyle \limsup \frac{n}{r(x,n)}=\limsup \frac{v_{k+1}-1}{v_k}= \limsup \frac{v_{k+1}}{v_k}\cdot$
\end{center}

Under the assumption $0<b_{i}< a_i-1$ for all $i\geq 1$ large enough, added with the values of the repetition function we obtained, we get
\begin{center}
	$\displaystyle \limsup \frac{n}{r(x,n)}= \limsup\left(\frac{q_{n+1}-\rho_{n+1}}{q_n}, \frac{q_{n+1}-b_{n+1}q_n}{q_{n+1}-\rho_{n+1}}, \frac{q_{n+1}-\rho_{n+1}+q_n}{q_{n+1}-b_{n+1}q_n},\frac{q_{n+1}}{q_{n+1}-\rho_{n+1}+q_n} \right)$,
\end{center}
hence the theorem. \end{proof}

\begin{coro}
Let $x= T^{\rho}(c_\alpha)$ be a sturmian word of slope $\alpha$, where $\alpha$ has continuants $(q_n)_{n\geq -1}$, and $\rho=\sum_{i\geq 0}b_{i+1}q_i$ is an $\alpha$-number of the slope $\alpha$. Then :
\begin{itemize}
	\item if $b_i<a_i$ for all $i\geq 1$ large enough, then
	\begin{center}
		$dio(x)\leq dio(c_\alpha)$,
	\end{center}
	\item if $b_i<a_i$ for all $i\geq 1$ large enough, and if the set of indexes $i\geq 1$ such that $b_i=0$ has integer intervals of arbitrary large length, then
	\begin{center}
		$dio(x)=dio(c_\alpha)$.
	\end{center}
\end{itemize}
\end{coro}

\begin{proof}
For the first part, given the values of the repetition function we obtained, we get, for all $i\geq 1$ large enough :
\begin{center}
	$r(x,m)\geq q_i$ \quad for $q_i-1\leq m \leq q_{i+1}-2$.
\end{center}
This implies 
\begin{center}
	$\displaystyle \frac{m}{r(x,m)}\leq \frac{q_{i+1}}{q_i}$
\end{center}
and passing to the limsup, we get
\begin{center}
	$\displaystyle dio(x) = 1 + \limsup \frac{m}{r(x,m)} \leq 1+\limsup{\frac{q_{m+1}}{q_m}}  = dio(c_\alpha).$
\end{center}

For the second part, from the first one we have	$dio(x)\leq dio(c_\alpha)$, so that it is now enough to show the opposite inequality. Let $1\leq i \leq j$ be two natural numbers, and we assume that  $b_{k+1}=0$ for all $i\leq k \leq j$. Again from the values of the repetition function, we get
\begin{itemize}
	\item $r(x,m)=q_k$ \quad for  $q_k-1\leq m \leq q_{k+1}-\rho_i-2$,
	\item $r(x,m)=q_{k+1}-\rho_i$ \quad  for $q_{k+1}-\rho_i -1 \leq m \leq q_{k+1}-2$,
\end{itemize}
so that
\begin{center}
	$\displaystyle \frac{m}{r(x,m)}= \frac{q_{k+1}-\rho_i-2}{q_k-1}$.
\end{center}
To get the result, we show that the right-hand side of
\begin{center}
	$\displaystyle \frac{\rho_i}{q_k} < \frac{q_i}{q_k}$
\end{center}
can be made arbitrary small by bounded it from above by a quantity that only depends on $k-i$, which will allow us to conclude by taking the limsup under our assumptions.

To do that, we introduce the sequence $(u_n)_{n\geq 1}$ satisfying the Fibonacci recurrence, namely $u_{n+2}=u_{n+1}+u_n$ but with the starting values $u_0=q_{i-1}$ and $u_1=q_i$, and let $(F_n)_{n\geq 0}$ is the classical Fibonacci sequence with  $F_{n+2}=F_{n+1}+F_n$, $F_0=0$ and $F_1=1$. We can express $(u_n)_{n\geq 1}$ by
\begin{center}
	$u_n= q_iF_n+q_{i-1}F_{n-1}$.
\end{center}
we can then use induction to get the lower bound on $(q_{n})_{n\geq 0}$ by
\begin{center}
	$q_{n}\geq u_{n-i}$
\end{center}
with the initial condition $q_{i-1}\geq u_0$ and $q_i \geq u_1$, added with the trivial inequality $q_{n+2}\geq q_{n+1}+q_n$. This gives, by taking $k=n$,
\begin{center}
	$\displaystyle \frac{q_i}{q_k}  \leq \frac{q_i}{q_iF_{k-i}+q_{i-1}F_{k-i-1}}  \leq \frac{1}{F_{k-i}},$
\end{center}
allowing us to conclude. \end{proof}

\section{Factorisations of sturmian words}

In this third section of the paper, we develop the formalism of $\alpha$-number and formal intercepts. As a result we provide general results on factorisations of sturmian words, the form of these factorisation being naturally linked to the definition of $\alpha$-numbers. In the last part of this section, we study some operations on $\alpha$-numbers, that can be interpreted as division operations in the Ostrowski numeration system.

The problem of factorisations of sturmian words is a central problem since a factorisation is a very convenient way to understand an infinite word with a practical point of view. Factorisations for the characteristic word were known (see \cite{lothaire}), but for general sturmian words no general structure were known. For factorisation of sturmian words, see \cite{brown3, fraenkel2}. In the work of J. Peltom\"aki \cite{peltomaki2}, a "square root" map is defined on sturmian words, which relies on existence of factorisations of sturmian words made with words with specific properties. For the particular case of the Fibonacci word $f$, he obtains the formula :
\begin{center}
	$\displaystyle \sqrt{f}=\prod_{i\geq 0}\widetilde{f_{3i+2}}$
\end{center}
where $(f_n)$ is the standard associated sequence, and for which he checks by elementary means that this infinite word indeed defines a sturmian word. This formula can be put in parallel to our factorisations results and to our constructions in theorem \ref{deuxtorsion}.

\subsection{Equivalence of $\alpha$-numbers}

\begin{defi}\label{interceptequivalent}
Let $\rho=\sum_{i\geq 0}b_{i+1}q_i$ be an $\alpha$-number. We define the $\alpha$-number $\rho+1$ as the formal intercept associated to the sturmian word
\begin{center}
	$T(T^\rho(c_\alpha))$,
\end{center}
and this allows us to define $\rho + k$ as the formal intercept associated to $T^k(T^\rho(c_\alpha))$.
We say that two formal intercepts $\rho$ and $\gamma$ are equivalent if there exist two integers $k,l\geq 0$ such that $\rho+k=\gamma +l$, so that :
\begin{center}
	$\rho \equiv \gamma$ \quad $\Longleftrightarrow$ \quad $\exists k,l\geq 0, \ \rho+k=\gamma +l$,
\end{center}
We say that an $\alpha$-number is a natural integer if it is associated to a suffix of the characteristic word.
\end{defi}

\begin{prop}\label{classenulle}
Let $\rho= \sum_{i\geq 0} b_{i+1}q_i$ be an $\alpha$-number, and for all $n\geq 1$, let $\rho_n=\sum_{i=0}^{n-1}b_{i+1}q_i$. The following statements are equivalent :
\begin{enumerate}[i)]
	\item $\rho$ is equivalent to the zero $\alpha$-number,
	\item one of the two sequences $ (\rho_n)_{n\geq 0}$ or $ (q_n-\rho_n)_{n\geq 0}$ converges,
	\item we have one of the following alternatives :
	1) $\exists N\geq 1$ such that $b_i=0$ for all $i\geq N$,
		2) $ \exists N\geq 1$ such that $b_{2i}= a_{2i}$ and $b_{2i+1}=0$ for all $i\geq N$,
		3) $\exists N\geq 1$ such that $b_{2i}=0$ and $b_{2i+1}= a_{2i+1}$ for all $i\geq N$.
\end{enumerate}
\end{prop}

\begin{proof}

Recalling that $\rho$ is a natural integer if and only if $T^{\rho}(c_\alpha)$ is a suffix of the characteristic word, which happens if and only if the sequence $(\rho_n)_{n\geq 1}$ is eventually constant in view of theorem \ref{interceptsturmien}, and all these conditions are equivalent to the coefficients $(b_i)_{i\geq 1}$ being eventually zero.

On the other hand, if $x$ is a sturmian word with formal intercept $\rho= \sum_{i\geq 0} b_{i+1}q_i$, and if $k=\liminf (q_n-\rho_n)< +\infty$, then since the considered sequence has natural integer values, the value $k$ is attained infinitely many times. We can reduce to a corresponding set of indexes all having the same parity, and if for example $q_{2(n+1)}-\rho_{2(n+1)}=k$ happens infinitely many times with $n\geq 1$, then for such an integer $n$, since $x$ and $T^{\rho_{2(n+1)}}(c_\alpha)$ share the same prefix of length $q_{2n+3}+q_{2n+2}-2-\rho_{2n+3}\geq q_{2(n+1)}+q_{2n+1}-2-\rho_{2n+2} =q_{2n+1}-2+k$ according to proposition \ref{prefixecommun}, the words $T^k(x)$ and $T^{\rho_{2(n+1)}+k}(c_\alpha)=T^{q_{2(n+1)}}(c_\alpha)$ share the same prefixes of length $q_{2n+1}-2$, which is the corresponding prefix of $c_\alpha$. This quantity going toward $+\infty$ with $n$, we deduce that $T^k(x)=c_\alpha$, and that $\rho$ has equivalence class zero.

Also, for $k\geq 0$, the sequences $(q_{2\left\lfloor \frac{n}{2}\right\rfloor}-1-k)_{n\geq 1}$ and $(q_{2\left\lfloor \frac{n}{2}\right\rfloor+1}-1-k)_{n\geq 1}$ are the formal intercepts of words of whom the characteristic word is a strict suffix according to proposition \ref{interceptmoinsun}. Seeing that for these two sequences, the coefficients $(b_i)_{i \geq1}$ satisfy the conditions of the proposition, we obtain the result.

\end{proof}

\begin{prop}\label{interceptincrementation}
Let $\rho= \sum_{i\geq 0}b_{i+1}q_i$ be an $\alpha$-number with non-zero equivalence class. Then the $\alpha$-number $\gamma=\rho+k$ is asymptotically given by the formula $\gamma_n=\rho_n+k$ for $n$ large enough. In other words, for all $k\geq 0$, there exists $N\geq 0$ such that for all $n\geq N$,
\begin{center}
	$\Psi_n(\rho +k)=\Psi_n(\rho)+k$.
\end{center}
This property is still true when $\rho$ has equivalence class zero under the assumption that $\rho +k$ is not a natural integer.
\end{prop}

\begin{proof}
Let $x=T^\rho(c_\alpha)$ be the sturmian word with formal intercept $\rho$, according to theorem \ref{interceptsturmien}. Since $\rho$ has non-zero equivalence class, the sequence $(q_n-\Psi_n(\rho))_{n\geq 1}$ tends towards infinity according to proposition \ref{classenulle}, and there exists $N\geq 1$ such that $q_n > \Psi_n(\rho)+k$ for all $n\geq N$. The words $T^{\Psi_{n}(\rho+k)}(c_\alpha)$ and $T^k(T^{\rho}(c_\alpha))$ share the same prefix of length $q_{n+1}+q_n -2 -\Psi_{n+1}(\rho) - k\geq q_{n}-1$ according to proposition \ref{prefixecommun}, so that $T^{\Psi_{n}(\rho+k)}(c_\alpha)$ and $T^k(T^{\Psi_{n}(\rho)}(c_\alpha))$ have a prefix of length at least $q_{n}-1$ in common, and so $\Psi_n(\rho+k)=\Psi_{n}(\rho)+k$ for all $n\geq N$ according to the bijection between $\alpha$-numbers and sturmian words given in theorem \ref{interceptsturmien}.

For the case where we only assume that $\rho +k$ is not a natural integer but has equivalence class zero, we apply the remark stated after the proposition \ref{interceptmoinsun}.
\end{proof}

\begin{prop}\label{interceptidentification}
Let $\rho= \sum_{i\geq 0}b_{i+1}q_i$ and $\gamma= \sum_{i\geq 0} c_{i+1}q_i$ be two $\alpha$-numbers none of which has equivalence class zero. Then $\rho$ and $\gamma$ are equivalent if and only if, there exists $N\geq 1$ such that 
\begin{center}
	$b_i=c_i$ \quad for all $i\geq N$.
\end{center}
\end{prop}

\begin{proof}
Let $k\geq 0$ and $l\geq 0$ be such that $\rho+k=\gamma+l$ from definition \ref{interceptequivalent}, and let $N\geq 1$ be such that for all $n\geq N$, $\Psi_n(\rho+k)=\Psi_n(\rho)+k$
and $\Psi_n(\gamma+l)=\Psi_n(\gamma)+l$. Then, by the formulas
\begin{center}
	$b_{n+1}q_n =\Psi_{n+1}(\rho)-\Psi_n(\rho)  = \Psi_{n+1}(\rho+k)-\Psi_n(\rho+k) = \Psi_{n+1}(\gamma+l)-\Psi_n(\gamma+l) = \Psi_{n+1}(\gamma)-\Psi_n(\gamma) =c_{n+1}q_n$
\end{center}
we deduce that $b_i=c_i$ for all $i$ large enough.

Conversely, if $b_i=c_i$ for all $i$ large enough, then for $n$ large enough, we have
\begin{center}
	$\Psi_{n+1}(\rho)-\Psi_n(\rho)=\Psi_{n+1}(\gamma)-\Psi_n(\gamma)$
\end{center}
which implies that the sequences $(\Psi_n(\rho))_{n\geq 1}$ and $(\Psi_n(\gamma))_{n\geq 1}$ differ asymptotically by a constant, and this implies that the $\alpha$-numbers $\rho$ and $\gamma$ are equivalent.
\end{proof}

\begin{defi}\label{interceptsupport}
Let $\rho=\sum_{i\geq 0}b_{i+1}q_i$ be an $\alpha$-number that is not a natural integer. We define the support $\textup{Supp}(\rho)$ of $\rho$ as 
\begin{center}
	$\textup{Supp}(\rho)=\{n\geq 0 \ |\ b_{n+1}\neq 0 \}$
\end{center}
and the function $\Lambda_{\rho}$ as
\begin{center}
	$\Lambda_{\rho}(n)=\min ( \textup{Supp}(\rho)\cap [n,+\infty[)$
\end{center}
which is always defined since $\rho$ is not a natural integer.
\end{defi}

\begin{prop}\label{inegalitesupport}
Let $n \geq 0$, and $\rho$ be an $\alpha$-number. Then
\begin{enumerate}
	\item $n\in \textup{Supp}(\rho)$ if and only if $\Psi_{n+1}(\rho)\geq q_n$, in particular we have the inequalities :
	\begin{center}
		$\Psi_{\Lambda_{\rho}(n)+1}(\rho)\geq q_{\Lambda_{\rho}(n)} \geq q_n$,
	\end{center}
	\item if $n\in \textup{Supp}(\rho)$, then $b_{n+2}\neq a_{n+2}$, in particular we have the inequalities :
	\begin{center}
		$\Psi_{\Lambda_{\rho}(n)+2}(\rho)< q_{\Lambda_{\rho}(n)+2}-q_{\Lambda_{\rho}(n)}=a_{\Lambda_{\rho}(n)+2}q_{\Lambda_{\rho}(n)+1}$,
	\end{center}
	\item $\Psi_{\Lambda_{\rho}(n)+1}(\rho)= b_{\Lambda_{\rho}(n)+1}q_{\Lambda_{\rho}(n)}+\Psi_{n}(\rho)$.
\end{enumerate}
\end{prop}

\begin{proof}
We write the $\alpha$-number $\rho$ in the form $\rho=\sum_{i\geq 0} b_{i+1}q_i$, where the coefficients $(b_i)_{i\geq 1}$ satisfy the Ostrowski conditions.

1) According to the definition \ref{interceptsupport}, $n \in \textup{Supp}(\rho)$ if and only if $b_{n+1}\neq 0$, so by characterisations of the Ostrowski conditions, if and only if $\Psi_{n+1}(\rho)=\sum_{i=0}^{n}b_{i+1}q_i \geq q_n$. For the second statement, we apply the result to the integer $\Lambda_{\rho}(n)$ which, according to definition \ref{interceptsupport} belongs to the support of $\rho$. The inequality on the right being consequence of the increasing of the sequence $(q_n)_{n\geq -1}$ and the trivial inequality $\Lambda_{\rho}(n)\geq n$.

2) It is clear that if $n\in \textup{Supp}(\rho)$, then $b_{n+2}\neq a_{n+2}$. The inequality is obtained by
\begin{align*}
	\Psi_{\Lambda_{\rho}(n)+2}(\rho)&=\displaystyle \sum_{i=0}^{\Lambda_{\rho}(n)+1} b_{i+1}q_i =\sum_{i=0}^{\Lambda_{\rho}(n)} b_{i+1}q_i + b_{\Lambda_{\rho}(n)+2}q_{\Lambda_{\rho}(n)+1} \\ &< q_{\Lambda_{\rho}(n)+1} + b_{\Lambda_{\rho}(n)+2}q_{\Lambda_{\rho}(n)+1} \\ &< (b_{\Lambda_{\rho}(n)+2}+1)q_{\Lambda_{\rho}(n)+1}\\ & \leq a_{\Lambda_{\rho}(n)+2}q_{\Lambda_{\rho}(n)+1}=q_{\Lambda_{\rho}(n)+2}-q_{\Lambda_{\rho}(n)} 
\end{align*}
and by characterisations of the Ostrowski conditions.

3) We write :
\begin{align*}
	\displaystyle \Psi_{\Lambda_{\rho}(n)+1}(\rho)&=\displaystyle \sum_{i=0}^{\Lambda_{\rho}(n)} b_{i+1}q_i = b_{\Lambda_{\rho}(n)+1}q_{\Lambda_{\rho}(n)}+\Psi_{\Lambda_{\rho}(n)}(\rho)\\ &= b_{\Lambda_{\rho}(n)+1}q_{\Lambda_{\rho}(n)}+\Psi_{n}(\rho)
\end{align*}
according to definition \ref{interceptsupport}.

\end{proof}

\subsection{Infinite products and complementation}

For $m\geq 1$ an integer, written $m=\sum_{i=0}^Nb_{i+1}q_i$ where $(b_i)_{i=1}^{N+1}$ satisfies the Ostrowski conditions, the prefix of length $m$ of the characteristic word $c_\alpha$ is known to be given by
\begin{center}
	$\displaystyle \mathbb{P}_m(c_\alpha)=\sideset{}{^\downarrow}\prod_{i=0}^Ns_i^{b_{i+1}}=s_{N}^{b_{N+1}}s_{N-1}^{b_N}\cdots s_0^{b_1}$,
\end{center}
see \cite{shallitauto}, theorem 9.1.13.

\begin{coro}\label{prefixesuffixefini}
Let $m\geq 1$ and $p\geq 1$ be two integers and $N\geq 1$ such that $m+p=q_{N+1}-2$, with $m= \sum_{i=0}^Nb_{i+1}q_i$  and  $p= \sum_{i=0}^Nc_{i+1}q_i$ where $(b_i)_{i=1}^{N+1}$ and $(c_i)_{i=1}^{N+1}$ satisfy the Ostrowski conditions. Then we have the product formula :
\begin{center}
	$s_{N+1}^{--} =\displaystyle \sideset{}{^\downarrow}\prod_{i=0}^Ns_i^{b_{i+1}} \cdot \sideset{}{^\uparrow}\prod_{i=0}^N \widetilde{s_{i}}^{c_{i+1}}$.
\end{center}
In particular, the word
\begin{center}
	$\displaystyle \widetilde{\mathbb{P}}_{p}(c_\alpha)=\sideset{}{^\uparrow}\prod_{i=0}^N \widetilde{s_{i}}^{c_{i+1}}= \widetilde{s_{0}}^{c_1}\widetilde{s_{1}}^{c_2}\cdots \widetilde{s_{N}}^{c_{N+1}}$
\end{center}
is a factor of $c_\alpha$.
\end{coro}

\begin{proof}
Since $s_{N+1}^{--}$ is palindromic, we have :
\begin{center}
	$s_{N+1}^{--}=\mathbb{P}_m(c_\alpha)\widetilde{\mathbb{P}}_{p}(c_\alpha)$
\end{center}
where the result comes from.
\end{proof}

The corollary \ref{prefixesuffixefini} concerns finite words of the form
\begin{center}
$\displaystyle \widetilde{\mathbb{P}}_{m}(c_\alpha)=\sideset{}{^\uparrow}\prod_{i=0}^N \widetilde{s_{i}}^{b_{i+1}}= \widetilde{s_{0}}^{b_1}\widetilde{s_{1}}^{b_2}\cdots \widetilde{s_{N}}^{b_{N+1}}$
\end{center}
where $m= \sum_{i=0}^Nb_{i+1}q_i$ is an integer written in the Ostrowski numeration system. In view of the nature of the product, it is natural to consider infinite words of the form
\begin{center}
	$\displaystyle \prod_{i=0}^{+\infty} \widetilde{s_{i}}^{b_{i+1}}$
\end{center}
where the sequence  $(b_i)_{i\geq 1}$ satisfies the Ostrowski conditions. This leads to the following definition.

\begin{defi}\label{defprho}
Let $\rho=\sum_{i\geq 0}b_{i+1}q_i$ be an $\alpha$-number that is not a natural integer. We define the sturmian word $\widetilde{\mathbb{P}}_{\rho}(c_\alpha)$ as the infinite product :
\begin{center}
	$\displaystyle \widetilde{\mathbb{P}}_{\rho}(c_\alpha)=\prod_{i=0}^{+\infty} \widetilde{s_{i}}^{b_{i+1}}=\lim_{N\rightarrow +\infty} \sideset{}{^\uparrow}\prod_{i=0}^N \widetilde{s_{i}}^{b_{i+1}}=\lim_{N\rightarrow +\infty} \widetilde{\mathbb{P}}_{\rho_N}(c_\alpha).$
\end{center}
\end{defi}

\begin{prop}\label{interceptproduit}
Let $\rho$ be an $\alpha$-number of the slope $\alpha$, that is not a natural integer, and such that $\rho\notin \{ \sigma_0, \sigma_1 \}$. The formal intercept of the sturmian word $\widetilde{\mathbb{P}}_{\rho}(c_\alpha)$  is asymptotically given by the sequence
\begin{center}
	$\displaystyle \overline{\rho}=(\Psi_n(q_{\Lambda_{\rho}(n)+1}-2-\rho_{\Lambda_{\rho}(n)+1}))_{n\geq 0}$.
\end{center}
We call this $\alpha$-number the complement of $\rho$.
\end{prop}

\begin{proof}
Let $n\geq 0$ be such that $\rho_{\Lambda_{\rho}(n)+1}\leq q_{\Lambda_{\rho}(n)+1}-2$. The prefix of length $\rho_{\Lambda_{\rho}(n)+1}$ of $\widetilde{\mathbb{P}}_{\rho}(c_\alpha)$ is given by
\begin{center}
	$\displaystyle \mathbb{P}_{\rho_{\Lambda_{\rho}(n)+1}}(\widetilde{\mathbb{P}}_{\rho}(c_\alpha))=\sideset{}{^\uparrow}\prod_{i=0}^{\Lambda_{\rho}(n)} \widetilde{s_{i}}^{b_{i+1}}$
\end{center}
according to corollary \ref{prefixesuffixefini}. Since $\rho_{\Lambda_{\rho}(n)+1}\geq q_{\Lambda_{\rho}(n)}$ in view of proposition \ref{inegalitesupport}, we see that the two words
\begin{center}
	$T^{q_{\Lambda_{\rho}(n)+1}-2-\rho_{\Lambda_{\rho}(n)+1}}(c_\alpha)$ \quad and \quad $\widetilde{\mathbb{P}}_{\rho}(c_\alpha)$
\end{center}
share the same prefix of length $q_{\Lambda_{\rho}(n)}-1$. Plus, with the general inequality
\begin{center}
	$q_{\Lambda_{\rho}(n)+1}-2-\rho_{\Lambda_{\rho}(n)+1} < q_{\Lambda_{\rho}(n)+1}$
\end{center}
seen in proposition \ref{inegalitesupport}, and from the proof of theorem \ref{interceptsturmien}, the $\Lambda_{\rho}(n)$-th term of the formal intercept of $\widetilde{\mathbb{P}}_{\rho}(c_\alpha)$ is given by the reduction modulo $q_{\Lambda_{\rho}(n)}$ of $q_{\Lambda_{\rho}(n)+1}-2-\rho_{\Lambda_{\rho}(n)+1}$, which equals $\Psi_{\Lambda_{\rho}(n)}(q_{\Lambda_{\rho}(n)+1}-2-\rho_{\Lambda_{\rho}(n)+1})$. The results follows since $n\leq \Lambda_{\rho}(n)$ thanks to proposition \ref{inegalitesupport}.
\end{proof}

As a consequence, for all $M\in \textup{Supp}(\rho)\cap [N, +\infty[$, we have :
\begin{center}
	$\Psi_N(q_{\Lambda_{\rho}(N)+1}-2-\rho_{\Lambda_{\rho}(N)+1})=\Psi_N(q_{\Lambda_{\rho}(M)+1}-2-\rho_{\Lambda_{\rho}(M)+1})$.
\end{center}
The following corollary is a consequence of the previous proposition applied to the two sturmian words $10c_\alpha$ and $01c_\alpha$, to see that the complement of their formal intercept are both zero.

\begin{coro}\label{factorisationcarac}
If $a_1\geq 2$, then :
\begin{center}
	$\displaystyle c_\alpha=\widetilde{s_0}^{a_1-2}\prod_{i\geq 1}\widetilde{s_{2i}}^{a_{2i+1}}$ \quad and \quad $\displaystyle c_\alpha=\widetilde{s_0}^{a_1-1}\widetilde{s_1}^{a_2-1}\prod_{i\geq 1}\widetilde{s_{2i+1}}^{a_{2i+2}}$.
\end{center}
If $a_1=1$ and $a_2\geq 2$, then :
\begin{center}
	$\displaystyle c_\alpha=\widetilde{s_1}^{a_2-2}\widetilde{s_2}^{a_3-1}\prod_{i\geq 2}\widetilde{s_{2i}}^{a_{2i+1}}$ \quad and \quad $\displaystyle c_\alpha=\widetilde{s_1}^{a_2-2}\prod_{i\geq 1}\widetilde{s_{2i+1}}^{a_{2i+2}}$
\end{center}
If $a_1=1$ and $a_2=1$, we have :
\begin{center}
	$\displaystyle c_\alpha=\widetilde{s_2}^{a_3-1}\prod_{i\geq 2}\widetilde{s_{2i}}^{a_{2i+1}}$ \quad and \quad $\displaystyle c_\alpha=\prod_{i\geq 1}\widetilde{s_{2i+1}}^{a_{2i+2}}$.
\end{center}
\end{coro}

\subsection{Factorisations of sturmian words}

In this subsection we give our main results on factorisations of sturmian words, see theorem \ref{theofactorisation}, and we start by a technical lemma.

\begin{lem}\label{complementairenoyau}
If $\rho$ has a non-zero equivalence class, then $\overline{\rho}$ has a non-zero equivalence class.
\end{lem}

\begin{proof}
We suppose by contradiction that $\overline{\rho}$ is equivalent to zero. Then one of the two sequences $(\Psi_n(\overline{\rho}))_{n\geq 1}$ or $(q_n-\Psi_n(\overline{\rho}))_{n\geq 1}$ converges according to proposition \ref{classenulle}.

In the case where $(\Psi_n(\overline{\rho}))_{n\geq 1}$ converges, this implies that $\overline{\rho}$ is a natural integer, and so there exists $k\geq 0$ and $N\geq 0$ such that $\Psi_n(q_{\Lambda_{\rho}(n)+1}-2-\Psi_{\Lambda_{\rho}(n)+1}(\rho))=k$ for all $n\geq N$ according to proposition \ref{interceptproduit}. For all $n \geq N$, there exists a coefficient $c_{n+1}$ with $0\leq c_{\Lambda_{\rho}(n)+1} \leq a_{\Lambda_{\rho}(n)+1}$ such that
\begin{center}
	$q_{\Lambda_{\rho}(n)+1}-2-\Psi_{\Lambda_{\rho}(n)+1}(\rho)=c_{\Lambda_\rho(n)+1}q_{\Lambda_\rho(n)} + k$.
\end{center}
according to definition \ref{interceptdef}.

If the coefficient $c_{\Lambda_\rho(n)+1}$ is zero for infinitely many $n\geq N$, then we have the equality
\begin{center}
	$\Psi_{\Lambda_{\rho}(n)+1}(\rho)=q_{\Lambda_{\rho}(n)+1}-2-k$
\end{center}
according to definition \ref{interceptdef}, for infinitely many $n\geq N$. In view of definition \ref{interceptdef}, definitions \ref{interceptequivalent} and proposition \ref{classenulle}, this implies that $\rho$ is equivalent to the zero $\alpha$-number, which is a contradiction with our hypothesis. In the case where infinitely many coefficients $c_{n+1}$ are non-zero, then since that for all $n\geq N$, we have :
\begin{center}
	$q_{\Lambda_{\rho}(n)+1}-2-\Psi_{\Lambda_{\rho}(n)+1}(\rho)=c_{\Lambda_{\rho}(n)+1}q_{\Lambda_{\rho}(n)}+k$
\end{center}
and since $\rho$ is not equivalent to the zero $\alpha$-number by assumption, for $n$ large enough we have :
\begin{align*}
	(a_{\Lambda_\rho(n)+1}-c_{\Lambda_\rho(n)+1})q_{\Lambda_\rho(n)}+q_{\Lambda_\rho(n)-1}&= \Psi_{\Lambda_\rho(n)+1}(\rho)+k+2\\&=\Psi_{\Lambda_\rho(n)+1}(\rho+k+2)
\end{align*}
according to proposition \ref{interceptincrementation}, the right-hand side of the above being written in the Ostrowski numeration system. But this is impossible because the sequence $(a_{\Lambda_\rho(n)+1}-c_{\Lambda_\rho(n)+1})q_{\Lambda_\rho(n)}+q_{\Lambda_\rho(n)-1}$ does not define an $\alpha$-number asymptotically, since the supports in the Ostrowski numeration system of these integers have a minimum that tends towards infinity with $n$.

In the case where $(q_n-\Psi_n(\overline{\rho}))_{n\geq 1}$ converges, meaning that $c_\alpha$ is a suffix of $T^{\overline{\rho}}(c_\alpha)$, then there exists $k\geq 0$ and $N\geq 0$ such that we have the alternative :
\begin{center}
	$\Psi_n(\overline{\rho})=q_{2\left\lfloor \frac{n}{2}\right\rfloor}-1-k$
\end{center}
for all $n\geq N$, or
\begin{center}
	$\Psi_n(\overline{\rho})=q_{2\left\lfloor \frac{n}{2}\right\rfloor+1}-1-k$
\end{center}
for all $n\geq N$. We can assume without loss of generality that we are in the first case, and so :
\begin{center}
	$\Psi_{n}(\overline{\rho})= q_{2\left\lfloor \frac{n}{2}\right\rfloor}-1-k=\Psi_{n}(q_{\Lambda_{\rho}(n)+1}-2-\Psi_{\Lambda_{\rho}(n)+1}(\rho))$
\end{center}
for some $k\geq 0$ and all $n\geq 1$ large enough, according to proposition \ref{interceptmoinsun}. As before we write, with help of proposition \ref{interceptproduit},
\begin{center}
	$\Psi_{\Lambda_{\rho}(n)}(\overline{\rho})=q_{\Lambda_{\rho}(n)+1}-2 -\Psi_{\Lambda_{\rho}(n)+1}(\rho)-c_{\Lambda_{\rho}(n)+1}q_{\Lambda_{\rho}(n)}$
\end{center}
leading to the equality
\begin{center}
	$ (a_{\Lambda_{\rho}(n)+1} -c_{\Lambda_{\rho}(n)+1})q_{\Lambda_{\rho}(n)}+q_{\Lambda_{\rho}(n)-1}+k=q_{2\left\lfloor \frac{n}{2}\right\rfloor}+1+\Psi_{\Lambda_{\rho}(n)+1}(\rho)$.
\end{center}

If $c_{\Lambda_{\rho}(n)+1}=0$ infinitely many times with $n$, then for such $n$ we have :
\begin{center}
	$q_{\Lambda_{\rho}(n)+1}-q_{2\left\lfloor \frac{n}{2}\right\rfloor}+k=\Psi_{\Lambda_{\rho}(n)+1}(\rho+1)$
\end{center}
which is a contradiction since the left-hand side does not define an $\alpha$-number according to definition \ref{interceptdef}. Indeed, if $M$ is an integer such that $k< q_M$, then the expansion of these integer in the Ostrowski numeration system have no support within the integer interval $[M, 2\left\lfloor \frac{n}{2}\right\rfloor-1[$, although $n$ can be chosen arbitrary large. If $c_{\Lambda_{\rho}(n)+1}\neq 0$ infinitely many times with $n$, then 
\begin{center}
	 $(a_{\Lambda_{\rho}(n)+1} -c_{\Lambda_{\rho}(n)+1})q_{\Lambda_{\rho}(n)}+q_{\Lambda_{\rho}(n)-1}-q_{2\left\lfloor \frac{n}{2}\right\rfloor}+k= \Psi_{\Lambda_{\rho}(n)+1}(\rho+1)$
\end{center}
and the left-hand side is an integer with an expansion in the Ostrowski numeration system having a support within the integer interval $[M, 2\left\lfloor \frac{n}{2}\right\rfloor[$ and we conclude as before.

\end{proof}

\begin{theo}\label{prefixesuffixesturmien}
The application $\rho \longmapsto \overline{\rho}$, from the set of $\alpha$-numbers with non-zero equivalence class to itself is an involution, and we have the formula
\begin{center}
	$T^\rho(c_\alpha)=\widetilde{\mathbb{P}}_{\overline{\rho}}(c_\alpha)$.
\end{center}
Moreover, the bi-infinite word
\begin{center}
	$\displaystyle \widetilde{T^{\overline{\rho}}(c_\alpha)}\cdot T^\rho(c_\alpha)$
\end{center}
is a sturmian orbit.
\end{theo}

\begin{proof}
Let $\rho$ be an $\alpha$-number with non-zero equivalence class. For $N\geq 1$, according to corollary \ref{prefixesuffixefini}, we have the factorisation
\begin{center}
	$s_{\Lambda_\rho(N)+1}^{--}=\mathbb{P}_{q_{\Lambda_{\rho}(N)+1}-2-\Psi_{\Lambda_{\rho}(N)+1}(\rho)}(c_\alpha)\widetilde{\mathbb{P}}_{\Psi_{\Lambda_{\rho}(N)+1}(\rho)}(c_\alpha)$
\end{center}
and since the word
\begin{center}
	$\mathbb{P}_{\Psi_n(q_{\Lambda_{\rho}(N)+1}-2-\Psi_{\Lambda_{\rho}(N)+1}(\rho))}(c_\alpha)$
\end{center}
is a suffix of $\mathbb{P}_{q_{\Lambda_{\rho}(N)+1}-2-\Psi_{\Lambda_{\rho}(N)+1}(\rho)}(c_\alpha)$, we deduce that for all $N\geq 1$, the word
\begin{center}
	$\mathbb{P}_{\Psi_N(q_{\Lambda_{\rho}(N)+1}-2-\Psi_{\Lambda_{\rho}(N)+1}(\rho))}(c_\alpha)\widetilde{\mathbb{P}}_{\Psi_{\Lambda_{\rho}(N)+1}(\rho)}(c_\alpha)$
\end{center}
is a factor of $c_\alpha$. Since the sequence $(\Psi_n(q_{\Lambda_{\rho}(n)+1}-2-\Psi_{\Lambda_{\rho}(N)+1}(\rho)))_{n\geq 1}$ defines an $\alpha$-number with non-zero equivalence class from proposition \ref{interceptproduit} and lemma \ref{complementairenoyau}, we deduce that the bi-infinite word
\begin{center}
	$\displaystyle \widetilde{\widetilde{\mathbb{P}}_{\overline{\rho}}(c_\alpha)}\cdot \widetilde{\mathbb{P}}_{\rho}(c_\alpha)$
\end{center}
is a sturmian orbit.

On the other hand, for $N\geq 1$, we know that the two words $T^{\Psi_{\Lambda_{\rho}(N)+1}(\rho)}(c_\alpha)$ and $T^{\rho}(c_\alpha)$ share the same prefix of length $q_{\Lambda_{\rho}(N)+2}+q_{\Lambda_{\rho}(N)+1}-2 -\Psi_{\Lambda_{\rho}(N)+2}(\rho) \geq q_{\Lambda_{\rho}(N)+1}-2 -\Psi_{\Lambda_{\rho}(N)+1}(\rho)$ according to proposition \ref{prefixecommun}. Since
\begin{center}
	$\mathbb{P}_{q_{\Lambda_{\rho}(N)+1}-2 -\Psi_{\Lambda_{\rho}(N)+1}(\rho)}(T^{\Psi_{\Lambda_{\rho}(N)+1}(\rho)}(c_\alpha))=\widetilde{\mathbb{P}}_{q_{\Lambda_{\rho}(N)+1}-2 -\Psi_{\Lambda_{\rho}(N)+1}(\rho)}(c_\alpha)$
\end{center}
from corollary \ref{prefixesuffixefini}, we get that the word
\begin{center}
	$\widetilde{\mathbb{P}}_{\Psi_{\Lambda_{\rho}(N)}(q_{\Lambda_{\rho}(N)+1}-2 -\Psi_{\Lambda_{\rho}(N)+1}(\rho))}(c_\alpha)$
\end{center}
is a prefix of $T^\rho(c_\alpha)$ from definition \ref{deftrho}, hence the relation
\begin{center}
	$T^\rho(c_\alpha)=\widetilde{\mathbb{P}}_{\overline{\rho}}(c_\alpha)$.
\end{center}
This relation implies that the correspondence $\rho \mapsto \overline{\rho}$ is an involution, according to theorem \ref{interceptsturmien}.
\end{proof}

\begin{coro}\label{complementairetheoreme}
Let $\rho=\sum_{i=0}^{+\infty}b_{i+1}q_i$ be an $\alpha$-number of the slope $\alpha$ with non-zero equivalence class. Then the three following $\alpha$-numbers are equal :
\begin{enumerate}
	\item the formal intercept associated to the unique sturmian word $y$ such that the bi-infinite word $\widetilde{y}\cdot T^{\rho}(c_\alpha)$ is a sturmian orbit,
	\item the formal intercept associated to the sturmian word
	\begin{center}
		$\displaystyle \widetilde{\mathbb{P}}_{\rho}(c_\alpha)=\prod_{i=0}^{+\infty}\widetilde{s_i}^{b_{i+1}}$,
	\end{center}
where $(s_n)_{n\geq -1}$ is the standard sequence with respect to the slope $\alpha$,
 \item the $\alpha$-number asymptotically defined by the sequence $\displaystyle \Psi_n(\overline{\rho})= \Psi_n(q_{\Lambda_{\rho}(n)+1}-2-\Psi_{\Lambda_{\rho}(n)+1}(\rho))$, where $\Lambda_{\rho}(n)=\min( \textup{Supp}(\rho)\cap [n, +\infty[)$ and the functions $\Psi_n : \mathcal{I}_\alpha \rightarrow [0,q_n[$ are the projections coming from the projective limit of definition \ref{interceptdef}.
\end{enumerate}
\end{coro}

\begin{coro}\label{factorisationsturmien}
Let $x$ be a sturmian word with non-zero equivalence class. Then there exists a unique $\alpha$-number $\rho$ with non-zero equivalence class such that :
\begin{center}
	$x=\widetilde{\mathbb{P}}_{\rho}(c_\alpha)$.
\end{center}
\end{coro}

\begin{prop}\label{produitplusun}
Let $\rho$ be a $\alpha$-number with non-zero equivalence class. Then we have the formula :
\begin{center}
	$T(\widetilde{\mathbb{P}}_{\rho+1}(c_\alpha))=\widetilde{\mathbb{P}}_{\rho}(c_\alpha)$.
\end{center}
In particular, for all $k\geq 0$ we have :
\begin{center}
	$\overline{\rho+k}+k=\overline{\rho}$.
\end{center}
In the case where $\rho$ has zero equivalence class, but is not a natural integer, then the first formula is valid if $\rho+1\neq 0$, and the second is valid if $\rho+k$ is not a natural integer.
\end{prop}

\begin{proof}
For the first part, it is an easy consequence of the fact that the two sturmian orbits
\begin{center}
	$	\displaystyle \widetilde{T^{\overline{\rho}}(c_\alpha)}\cdot T^\rho(c_\alpha)$ \quad and \quad $	\displaystyle \widetilde{T^{\overline{\rho+1}}(c_\alpha)}\cdot T^{\rho+1}(c_\alpha)$
\end{center}
are equal according to theorem \ref{prefixesuffixesturmien}.

For the second part, this is a consequence of the factorisation formulas obtained for the characteristic word in corollary \ref{factorisationcarac},  and the fact that if $\rho+k$ is not a natural integer, then $\Psi_N(\rho+k)=\Psi_N(\rho)+k$ for $N$ large enough according to proposition \ref{interceptincrementation}.
\end{proof}

\begin{prop}
Let $\sigma_0$ be the formal intercept associated by theorem \ref{interceptsturmien} to the word $0c_\alpha$, and let $\sigma_1$ be the formal intercept associated to the word $1c_\alpha$. Then we have the factorisations
\begin{align*}
	0c_\alpha=T^{\sigma_0}(c_\alpha)=\widetilde{\mathbb{P}}_{\sigma_1}(c_\alpha)
\end{align*}
and
\begin{align*}
	1c_\alpha=T^{\sigma_1}(c_\alpha)=\widetilde{\mathbb{P}}_{\sigma_0}(c_\alpha).
\end{align*}
\end{prop}

\begin{proof}
Let $\gamma_0$ be the formal intercept associated to the word $10c_\alpha$, and let $\gamma_1$ be the formal intercept associated to the word $01c_\alpha$. Corollary \ref{factorisationcarac} may be rewritten as the equalities :
\begin{center}
	$c_\alpha = \widetilde{\mathbb{P}}_{\gamma_0}(c_\alpha)=\widetilde{\mathbb{P}}_{\gamma_1}(c_\alpha) $
\end{center}
and given that $\gamma_0+1=\sigma_0$ and $\gamma_1+1=\sigma_1$, with proposition \ref{produitplusun}, the words
\begin{center}
	$\widetilde{\mathbb{P}}_{\sigma_0}(c_\alpha)$ \quad and \quad $\widetilde{\mathbb{P}}_{\sigma_1}(c_\alpha)$,
\end{center}
when deprived from their first letters, are equal to the characteristic word. We conclude as in proposition \ref{interceptmoinsun} by seeing that the word $\widetilde{\mathbb{P}}_{\sigma_1}(c_\alpha)$ starts with the letter 0, and the word $\widetilde{\mathbb{P}}_{\sigma_0}(c_\alpha)$ starts with the letter 1.
\end{proof}

\begin{prop}\label{factorisationsuffixcarac}
Let $x$ be a sturmian word of slope $\alpha$. The following statements are equivalent :
\begin{enumerate}[i)]
	\item $x$ is a suffix of the characteristic word,
	\item there are exactly two distincts $\alpha$-numbers $\rho$ and $\gamma$ such that
	\begin{align*}
		x&=\widetilde{\mathbb{P}}_{\rho}(c_\alpha)\\&=\widetilde{\mathbb{P}}_{\gamma}(c_\alpha).
	\end{align*}
\end{enumerate}
\end{prop}

\begin{proof}
We obtained two factorisations formulas for the characteristic word, given by :
\begin{align*}
	c_\alpha &=\widetilde{\mathbb{P}}_{\gamma_0}(c_\alpha)\\&=\widetilde{\mathbb{P}}_{\gamma_1}(c_\alpha)
\end{align*}
where $\gamma_0$ and $\gamma_1$ are the two $\alpha$-numbers such that $\gamma_0+2=0$ and $\gamma_1+2=0$. This shows that the suffixes of the characteristic word are given by the words
\begin{center}
	$x=\widetilde{\mathbb{P}}_{\rho}(c_\alpha)$
\end{center}
where $\rho$ runs through the set of $\alpha$-numbers satisfying $\rho+k=\gamma_0$ or $\rho+k=\gamma_1$ for some $k\geq 0$, this enumeration being exhaustive and without repetition. Having used every $\alpha$-numbers when this implication is proved and with in mind corollary \ref{factorisationsturmien}, the sturmian words satisfying $ii)$ also satisfy $i)$.
\end{proof}

\begin{prop}
Let $x$ be a sturmian word of slope $\alpha$. The following statements are equivalent :
\begin{enumerate}[i)]
	\item one of the two words $01c_\alpha$ or $10c_\alpha$ is a suffix of $x$,
	\item the word $x$ has no factorisation of the form
	\begin{center}
		$x=\widetilde{\mathbb{P}}_{\rho}(c_\alpha)$
	\end{center}
for an $\alpha$-number $\rho$.
\end{enumerate}
\end{prop}

\begin{proof}
We saw that a sturmian word associated to a formal intercept that has non-zero equivalence class cannot have the characteristic word as a suffix. On the other hand, in view of proposition \ref{factorisationsuffixcarac} above, the $\alpha$-numbers having zero class and that are not natural integers give infinite products by the definition \ref{defprho} exactly the sturmian words that are suffixes of the characteristic word or the two words $1c_\alpha$ and $0c_\alpha$. The proposition is a consequence of these characterisations and of the corollary \ref{factorisationsturmien}.

\end{proof}

\begin{theo}\label{theofactorisation}
Let $x$ be a sturmian word of slope $\alpha$, with $x\neq 0c_\alpha$ and $x\neq 1c_\alpha$. Then :
\begin{enumerate}
	\item the following statements are equivalent :
	\begin{itemize}
		\item the word $x$ and the characteristic word $c_\alpha$ have no suffix in common,
		\item there exist a unique $\alpha$-number $\rho$ of the slope $\alpha$ such that $x= \widetilde{\mathbb{P}}_{\rho}(c_\alpha) $,
	\end{itemize}
 \item the following statements are equivalent :
\begin{itemize}
		\item the word $x$ is a suffix of $c_\alpha$,
		\item there exists exactly two $\alpha$-numbers $\rho$ and $\gamma$ of the slope $\alpha$ such that $x= \widetilde{\mathbb{P}}_{\rho}(c_\alpha)=  \widetilde{\mathbb{P}}_{\gamma}(c_\alpha)$,
	\end{itemize}
	 \item the following statements are equivalent :
	\begin{itemize}
		\item one of the two words $10c_\alpha$ or $01c_\alpha$ is a suffix of $x$,
		\item the word $x$ has no factorisation of the form $x= \widetilde{\mathbb{P}}_{\rho}(c_\alpha) $ for an $\alpha$-number $\rho$.
	\end{itemize}
\end{enumerate}
The two exceptions $0c_\alpha$ and $1c_\alpha$ of respective formal intercepts $\sigma_0$ and $\sigma_1$ satisfy $0c_\alpha=\widetilde{\mathbb{P}}_{\sigma_1}(c_\alpha)$ and $1c_\alpha=\widetilde{\mathbb{P}}_{\sigma_0}(c_\alpha)$, and have no other factorisations of this form.
\end{theo}

\subsection{Torsion relations}

Consider the slope $\alpha=1/\varphi^2$ of the Fibonacci word $c_{\alpha}$, where $\varphi$ is the golden ratio. We denote $(q_n)_{n\geq -1}=(F_n)_{n\geq -1}$ its sequence of continuants, also known as the Fibonacci sequence,where $F_{-1}=0$, $F_0=1$ and $F_{n+2}=F_{n+1}+F_n$ for $n\geq -1$. The $\alpha$-numbers of this slope write uniquely in the form
\begin{center}
	$\rho=\displaystyle \sum_{i \geq 0}b_{i+1}F_i$
\end{center}
where the coefficients $(b_i)_{i\geq 1}$ equal 0 or 1, and are submitted to the condition $b_ib_{i+1}=0$ for all $i\geq 1$. consider, for $j\in \{0,1,2,3\}$, the $\alpha$-numbers, that are all non-equivalent to each other :
\begin{center}
	$\displaystyle\mathcal{F}_4^{(j)}=\sum_{i\geq 1}F_{4i+j}$,
\end{center}
and let us compute their complement. We have 
	$ \Lambda_{\mathcal{F}_4^{(j)}}(n)=4\left\lfloor \frac{n-j}{4}\right\rfloor+j$ for $j\in \{0,1,2,3\}$ and $n\geq 7$, providing :
\begin{center}
	$\displaystyle F_{4\left\lfloor \frac{n-j}{4}\right\rfloor+j+1}-2-\sum_{i=1}^{\left\lfloor \frac{n-j}{4}\right\rfloor}F_{4i+j}=m_j+\sum_{i=1}^{\left\lfloor \frac{n-j}{4}\right\rfloor-1}F_{4i+j+2}$
\end{center}
where $m_0=F_2$, $m_1=F_3$, $m_2=F_2+F_4$ and $m_3=F_3+F_5$, with in mind the classical formulas  $ \sum_{i=1}^NF_{2i}=F_{2N+1}-1$ and $ \sum_{i=0}^NF_{2i+1}=F_{2N+2}-1$ for $N\geq 1$. This shows :
\begin{center}
	$\overline{\mathcal{F}_4^{(0)}}=\mathcal{F}_4^{(2)}+F_2$, \quad $\overline{\mathcal{F}_4^{(1)}}=\mathcal{F}_4^{(3)}+F_3$, \quad $\overline{\mathcal{F}_4^{(2)}}=\mathcal{F}_4^{(0)}+F_2$, \quad and \quad $\overline{\mathcal{F}_4^{(3)}}=\mathcal{F}_4^{(1)}+F_3$.
\end{center}
If we consider the index $j\in \{0,1,2,3\}$ as an element of $\mathbb{Z}/4\mathbb{Z}$, then we've just see that the equivalence class of the complement of $\mathcal{F}_4^{(j)}$ is the class of $\mathcal{F}_4^{(j+2)}$.

Consider now the $\alpha$-numbers :
\begin{center}
	$\displaystyle \mathcal{F}_3^{(j)}=\sum_{i\geq 1}F_{3i+j}$
\end{center}
for $j\in \{0,1,2\}$. As before we compute $ \Lambda_{\mathcal{F}_3^{(j)}}(n)=3\left\lfloor \frac{n-j}{3}\right\rfloor+j$ for $n\geq 3$, and the computations :
\begin{align*}
  F_{3\left\lfloor \frac{n-j}{3}\right\rfloor+j+1}-\sum_{i=1}^{\left\lfloor \frac{n-j}{3}\right\rfloor} F_{3i+j}&= F_{3\left(\left\lfloor \frac{n-j}{3}\right\rfloor-1\right)+j+1}-\sum_{i=1}^{\left\lfloor \frac{n-j}{3}\right\rfloor-2} F_{3i+j} \\ & = F_{3\left(\left\lfloor \frac{n-j}{3}\right\rfloor-1\right)+j}+F_{3\left(\left\lfloor \frac{n-j}{3}\right\rfloor-2\right)+j+2}-\sum_{i=1}^{\left\lfloor \frac{n-j}{3}\right\rfloor-2} F_{3i+j} \\ &= F_{2+j} + \sum_{i=1}^{\left\lfloor \frac{n-j}{3}\right\rfloor-1} F_{3i+j}
\end{align*}
show that
\begin{center}
	$\overline{\mathcal{F}_3^{(0)}}=\mathcal{F}_3^{(0)}$, \quad $\overline{\mathcal{F}_3^{(1)}}=\mathcal{F}_3^{(1)}+F_1$, \quad and \quad $\overline{\mathcal{F}_3^{(2)}}=\mathcal{F}_3^{(2)}+F_3$.
\end{center}
In particular, the $\alpha$-numbers $\mathcal{F}_3^{(j)}$, for $j\in\{0,1,2\}$, are all equivalent to their complement.

\begin{theo}\label{complementfacile}
Let $\alpha=[0,a_1,a_2,\cdots]$ be a slope with continuants $(q_n)_{n\geq -1}$. Let $M\subset \mathbb{N}\backslash\{0,1\}$ be infinite with infinite complement in $\mathbb{N}\backslash\{0,1\}$. Then the complement of the $\alpha$-numbers of the slope $\alpha$
\begin{center}
	$\displaystyle \mathcal{F}_M^{(0)}=\sum_{i\in M}a_{2i+1}q_{2i}$ \quad and \quad $\displaystyle \mathcal{F}_M^{(1)}=\sum_{i\in M}a_{2i+2}q_{2i+1}$
\end{center}
are given by the following formulas, where $M^c$ denotes the complement of $M$ in $\mathbb{N}\backslash\{0,1\}$ :
\begin{center}
	$\displaystyle \overline{\mathcal{F}_M^{(0)}}=q_3-2 +\mathcal{F}_{M^c}^{(0)}$
\quad and \quad $\displaystyle \overline{\mathcal{F}_M^{(1)}}=q_4-2 +\mathcal{F}_{M^c}^{(1)}$.
\end{center}
\end{theo}

\begin{proof}
We give the proof for the computation of $\overline{\mathcal{F}_M^{(0)}}$, the second one following the same lines.

We first see that for all $m\in M$, we have
\begin{center}
	$\Lambda_{\mathcal{F}_M^{(0)}}(2m)=2m$
\end{center}
and the $\alpha$-number $\overline{\mathcal{F}_M^{(0)}}$ admits as $2m$-th component the integer
\begin{align*}
	\displaystyle \Psi_{2m}(q_{2m+1}-2- \Psi_{2m+1}(\mathcal{F}_M^{(0)}))&=\Psi_{2m}\left(q_{2m+1}-2- \sum_{i\in M, i\leq m}a_{2i+1}q_{2i}\right) \\ &=\Psi_{2m}\left(q_3-2+\sum_{i\in M^c, i\leq m}a_{2i+1}q_{2i}\right) \\ &=q_3-2+\sum_{i\in M^c, i\leq m-1}a_{2i+1}q_{2i}
\end{align*}
which allows us to conclude that $\overline{\mathcal{F}_M^{(0)}}=q_3-2 +\mathcal{F}_{M^c}^{(0)}$.
\end{proof}

\begin{prop}\label{complementpair}
Let $\alpha$ be a slope of partial quotients $(a_i)_{i \geq1}$ with continuants $(q_n)_{n\geq -1}$. We suppose that the integers $(a_i)_{i\geq 1}$ are eventually all even, from a index $2k_0$. Then the following three $\alpha$-numbers
\begin{center}
	$\displaystyle \mathcal{S}_0=\sum_{i\geq k_0} \frac{a_{2i+1}}{2}q_{2i}$, \quad $\displaystyle\mathcal{S}_1=\sum_{i\geq k_0} \frac{a_{2i+2}}{2}q_{2i+1}$, \quad and \quad $\displaystyle\mathcal{S}_2=\sum_{i\geq  2k_0} \frac{a_{i+1}}{2}q_{i}$
\end{center}
are $\alpha$-numbers of the slope $\alpha$ equivalent to their complement, and belong to different equivalence classes.
\end{prop}

\begin{proof}
The cases of $\mathcal{S}_0$ and $\mathcal{S}_1$ are similar, so we will restrict to the cases of $\mathcal{S}_0$ and  $\mathcal{S}_2$. The fact that they are not equivalent is a consequence of proposition \ref{interceptidentification}.

Let us proceed to the case of $\overline{\mathcal{S}_0}$. As for the proof of proposition \ref{complementfacile}, we have, for all $m\geq \left\lfloor k_0/2\right\rfloor+1$,
\begin{center}
	$\Lambda_{\mathcal{S}_0}(2m)=2m$
\end{center}
and we compute, for such a $m\geq k_0$,
\begin{align*}
	\displaystyle \Psi_{2m}(q_{2m+1}-2- \Psi_{2m+1}(\mathcal{S}_0))&=\Psi_{2m}\left(q_{2m+1}-2- \sum_{i =  k_0}^m\frac{a_{2i+1}}{2}q_{2i}\right) \\ &=\Psi_{2m}\left(q_{ 2k_0-1}-2+\sum_{i =  k_0}^m\frac{a_{2i+1}}{2}q_{2i}\right) \\ &=q_{ 2k_0-1}-2+\sum_{i =  k_0}^{m-1}\frac{a_{2i+1}}{2}q_{2i}
\end{align*}
given the classical formula
\begin{center}
	$q_{2m+1}-q_{2k_0-1}=a_{2m+1}q_{2m}+a_{2m-1}q_{2m-2}+\cdots + a_{2k_0+1}q_{2k_0}$
\end{center}
and given that the support in the Ostrowski numeration system of the integer $q_{ 2k_0-1}-2$ is contained in the integer interval $[1, 2k_0-1[$. This shows that
\begin{center}
	$\displaystyle \overline{\mathcal{S}_0}=q_{ 2k_0-1}-2+\sum_{i =  k_0}^{+\infty}\frac{a_{2i+1}}{2}q_{2i}$
\end{center}
which upon the use of proposition \ref{interceptidentification}, shows that $\mathcal{S}_0$ is equivalent to its complement.

For the case of $\overline{\mathcal{S}_2}$, we first see that $\Lambda_{\mathcal{S}_2}(m)=m$ for all $m\geq 2k_0$. We get
\begin{align*}
	\displaystyle \Psi_{m}(q_{m+1}-2- \Psi_{m+1}(\mathcal{S}_2))&=\Psi_{m}\left(q_{m+1}-2- \sum_{i =  2k_0}^{m}\frac{a_{i+1}}{2}q_{i}\right) \\ &=\Psi_{m}\left(\left(\frac{a_{m+1}}{2}-1\right)q_m+q_{2k_0-1}+q_{2k_0-2}-2+\sum_{i =  2k_0}^{m-1}\frac{a_{i+1}}{2}q_{i}\right) \\ &=q_{2k_0-1}+q_{2k_0-2}-2+\sum_{i =  2k_0}^{m-1}\frac{a_{i+1}}{2}q_{2i}
\end{align*}
given the formula
\begin{center}
	$q_{m+1}+q_m-q_{2k_0-1}-q_{2k_0-2}=a_{m+1}q_{m}+a_{m}q_{m-1}+\cdots + a_{2k_0+1}q_{2k_0}$
\end{center}
and the fact that the integer $q_{2k_0-2}-2$ has a support in the Ostrowski numeration system contained in the integer interval $[1,2k_0-2[$. This shows that 
\begin{center}
	$\displaystyle \overline{\mathcal{S}_2}=q_{2k_0-1}+q_{2k_0-2}-2+\sum_{i =  2k_0}^{+\infty}\frac{a_{i+1}}{2}q_{2i}$
\end{center}
which upon the use of proposition \ref{interceptidentification}, shows that $\mathcal{S}_2$ is equivalent to its complement.
\end{proof}

We now aim to construct $\alpha$-numbers equivalent to their complement for general slopes. Given the above result, we have to treat the general case where the coefficients $(a_i)_{i\geq 1}$ are not eventually even. We are going to use the following fomulas on the continuants $(q_i)_{\geq -1}$, valid for all $i\geq 0$ and all $k\geq 0$ :
\begin{center}
$q_{i+2}-q_{i}=a_{i+2}q_{i+1}$ \quad and \quad $\displaystyle q_{i+3+k}-q_{i}=(a_{i+3+k}-1)q_{i+2+k}+\sum_{l=1}^ka_{i+2+l}q_{i+1+l}+(a_{i+2}+1)q_{i+1}$	
\end{center}
which can be proved by easy inductions on $k\geq 1$, the first one being actually trivial. The first one will be used when $a_{i+2}$ is even, and the second one when $a_{i+2}$ and $a_{i+3+k}$ are odd and all the intermediate integers $(a_{i+2+l})_{l=1}^k$ are even.

We consider the set $\mathcal{B}$ of finite words over the alphabet $\{0,1\}$ given by :
\begin{center}
	$\mathcal{B}=\{00,01\}\cup \{10^k10, 10^k11 \ | \ k\geq 0\}$
\end{center}
and start with a lemma.

\begin{lem}\label{bfactorisationfini}
For $u$ a finite word over the alphabet $\{0,1\}$, we have :
\begin{enumerate}
	\item the word $u$ has at most one factorisation made of elements of $\mathcal{B}$,
	\item one and only one of the three words $u$, $1u$ and $11u$ has a factorisation made of elements of $\mathcal{B}$,
  \item every infinite word $y$ over the alphabet $\{0,1\}$ in which the letter 1 appears infinitely many times has a unique factorisation made of elements of $\mathcal{B}$.
	\end{enumerate}
\end{lem}

\begin{proof}
1) Since no elements of $\mathcal{B}$ is prefix of one another, if $u$ has a factorisation made of elements of $\mathcal{B}$, the first terms of these factorisations must agree. We conclude by omitting this common first term and using this same argument.

2) The statement is easily checked for finite words $u$ with length $|u|\leq 3$. We then proceed by induction on $|u|$, and we assume now $|u|>3$. If $u$ is a power of the letter 0, the result is easily checked. If $u$ ends with 00 or 01, then we can conclude by induction. If $u$ has at least three occurrences of the letter 1, then it admits a suffix belonging to the set $ \{10^k10, 10^k11 \ | \ k\geq 0\}$, and we can conclude by induction. Likewise if $u$ admits two occurrences of the letter 1 and ends with the letter 0. At last, if $u$ is of the form  $0^k11$ or $0^k10$ for some $k\geq0$, then the statement is easily checked.

3) From the assumption, at least one element of $\mathcal{B}$ is a prefix of $y$, and it is uniquely determined from the form of $\mathcal{B}$. By omitting this prefix from $y$, and using the same argument, we get existence ance unicity of such a factorisation. Notice that the infinite word $10000\cdots$ has no factorisation made of elements of $\mathcal{B}$.
\end{proof}

To construct $\alpha$-numbers equivalent to their complement, we are going to use the notion of factorisation of an indexed suffix of an infinite word $y$. This word is bound to be the infinite word with $i$-th letter being 0 or 1 depending on the parity of the integer $a_i$. A factorisation of an indexed suffix of $y$ is a couple $(n,(u_j)_{j\geq1})$  where $n\geq 0$ is an integer and $(u_j)_{j\geq1}$ is a sequence of words such that
\begin{center}
	$T^n(y)=u_1u_2u_3\cdots$.
\end{center}
and we will say that a factorisation is a $\mathcal{B}$-factorisation when all its terms belong to the set $\mathcal{B}$. We will say that two factorisations $(n,(u_j)_{j\geq1})$ and $(m,(v_j)_{j\geq1})$ of an indexed suffix of $y$, with $n\leq m$, are equivalent if there exists an integer $l\geq 0$ such that for all $j\geq 1$, $v_j=u_{j+l}$ and if $m=n+|u_1u_2\cdots u_{l-1}|$.

\begin{prop}\label{suffixeindexe}
Let $y$ be an infinite word over the alphabet $\mathcal{A}$. We assume that the letter 1 appears infinitely many times in $y$. Then the infinite word $y$ has exactly three equivalence classes of $\mathcal{B}$-factorisations of indexed suffixes.
\end{prop}

\begin{proof}
According to lemma \ref{bfactorisationfini}, the three words $y$, $1y$ and $11y$ admit a $\mathcal{B}$-factorisation. Write $y=u_1u_2u_3\cdots$, $1y=v_1v_2v_3\cdots$ and $11y=w_1w_2w_3\cdots$, where the words $(u_i)_{i\geq 1}$, $(v_i)_{i\geq 1}$ and $(w_i)_{i\geq 1}$ all belong to $\mathcal{B}$. The factorisations $(0,(u_i)_{i\geq 1})$, $(|v_1|,(v_i)_{i\geq 2})$ and $(|w_1|,(w_i)_{i\geq 2})$ are $\mathcal{B}$-factorisations of an indexed suffix of $y$, and are not equivalent from statement 2) of lemma \ref{bfactorisationfini} applied to a prefix of $y$. This same lemma guarantees that every $\mathcal{B}$-factorisation of an indexed suffix of $y$ is equivalent to one of this three $\mathcal{B}$-factorisations of an indexed suffix of $y$.
\end{proof}

\begin{theo}\label{deuxtorsion}
Let $\alpha$ be a slope with partial quotients $(a_i)_{i\geq 1}$, such that the coefficient $(a_i)_{i\geq 1}$ are not eventually even. Then there exists exactly three non-zero equivalence classes of $\alpha$-numbers of the slope $\alpha$ equivalent to their complement.
\end{theo}

\begin{proof}
We denote by $y=\overline{a_1}\overline{a_2}\overline{a_3}\cdots$ the infinite word whose $i$-th letter is the reduction modulo $2$ of the coefficient $a_i$, meaning that we denote, for $k\geq 1$, $\overline{k}=0$ if $k$ is even, and $\overline{k}=1$ if $k$ is odd. By assumption, in $y$ the letter 1 appears infinitely many times, and by the proposition \ref{suffixeindexe}, $y$ admits exactly three equivalence classes of $\mathcal{B}$-factorisations of an indexed suffix.

We will build the desired equivalence classes of $\alpha$-numbers associated to the slope $\alpha$ from the equivalence classes of $\mathcal{B}$-factorisations of an indexed suffix of $y$.

Let $i\geq 1$ and $u=\overline{a_{i+2}}\overline{a_{i+3}}$ be a factor of $y$ with length 2, such that $u\in \mathcal{B}$, meaning that $u=00$ or $u=01$, which means that $a_{i+2}$ is even. In this case we use the formula $q_{i+2}-q_{i}=a_{i+2}q_{i}$, that we write as the equality between integers
\begin{center}
	$\displaystyle \frac{q_{i+2}-q_{i}}{2}= \frac{a_{i+2}}{2}q_{i+1}$.
\end{center}
We now consider a factor $u=\overline{a_{i+2}}\overline{a_{i+3}}\cdots \overline{a_{i+2+k}}\overline{a_{i+3+k}}\overline{a_{i+4+k}}$ of $y$ with $|u|\geq 3$ and $k\geq 0$, such that $u\in \mathcal{B}$, meaning that $u=10^k10$ or $u=10^k11$. We use in this case the formula
\begin{center}
	$\displaystyle q_{i+3+k}-q_{i}=(a_{i+3+k}-1)q_{i+2+k}+\sum_{l=1}^ka_{i+2+l}q_{i+1+l}+(a_{i+2}+1)q_{i+1}$
\end{center}
that we write as the following equality, the right-hand side being written in the Ostrowski numeration system :
\begin{center}
	$\displaystyle \frac{q_{i+3+k}-q_{i}}{2}=\frac{a_{i+3+k}-1}{2}q_{i+2+k}+\sum_{l=1}^k\frac{a_{i+2+l}}{2}q_{i+1+l}+\frac{a_{i+2}+1}{2}q_{i+1}$.
\end{center}

Let now be $(n,(u_j)_{j\geq 1})$ a $\mathcal{B}$-factorisation of an indexed suffix of $y$. We write $T^n(y)=u_1u_2u_3\cdots= \overline{a_{n+1}}\overline{a_{n+2}}\overline{a_{n+3}}\cdots$, and we denote by  $(c_j)_{j\geq 1}$ the sequence of indexes where the terms of the factorisation appear, meaning that $c_j=n+1+|u_1u_2\cdots u_{j-1}|$ and so that $T^{c_j}(y)=u_{j}u_{j+1}u_{j+2}\cdots$. We saw that the integers, for all $1\leq j_1 < j_2$,
\begin{center}
	$\displaystyle \frac{q_{c_{j_2}}-q_{c_{j_1}}}{2}=\sum_{h={j_1}}^{{j_2}-1}\frac{q_{c_{h+1}}-q_{c_h}}{2}$
\end{center}
have their supports in the Ostrowski numeration system contained in the intervals $[c_{{j_1}}, c_{{j_2}}[$. This means, according to definition \ref{interceptdef}, that the infinite sum
\begin{center}
	$\displaystyle \rho =\sum_{j=1}^{+\infty}\frac{q_{c_{j+1}}-q_{c_j}}{2}$
\end{center}
defines an $\alpha$-number. This $\alpha$-number, constructed from a $\mathcal{B}$-factorisation of an indexed suffix of $y$, defines a sturmian word from theorem \ref{interceptsturmien}, obtained by glueing integers of the form $(q_{c_{j+1}}-q_{c_j})/2$, which is a legal operation in view of their support in the Ostrowski numeration system.

Two $\alpha$-numbers constructed that way are equivalent in the sense of definition \ref{interceptequivalent}  if and only if the corresponding $\mathcal{B}$-factorisations of an indexed suffix of $y$ are equivalent. By re-using the above notations, the sequence $(c_j)_{j\geq 1}$ can be recovered from the $\alpha$-number
\begin{center}
	$\displaystyle \rho =\sum_{k=1}^{+\infty}\frac{q_{c_{k+1}}-q_{c_k}}{2}$
\end{center}
written as $\rho=\sum_{i\geq 0}b_{i+1}q_i$ by consideration of the indexes $i\geq 1$ such that $b_{i+1}=(a_{i+1}-1)/2$, since these indexes correspond exactly to the indexes of the letter before last of the factors $u_j\in \mathcal{B}$ of the form $u=10^k10$ or $u=10^k11$ for some $k\geq 0$. The position of these indexes determines the equivalence class of a $\mathcal{B}$-factorisation of an indexed suffix of $y$, and we deduce from this that two such $\mathcal{B}$-factorisations that are not equivalent lead to the construction of two $\alpha$-numbers that are not equivalent. Conversely, two $\mathcal{B}$-factorisations of an indexed suffix of $y$ that are equivalent lead to $\alpha$-numbers defined by infinite sums with terms eventually equal. In view of the support of these terms in the Ostrowski numeration system, the $\alpha$-numbers constructed will have coefficients eventually equal, and hence will be equivalent in view of proposition \ref{interceptidentification}.

We hence have constructed 3 non-zero classes of $\alpha$-numbers, and we have to show now that they are equivalent to their complement, and we keep the notations introduced during this proof. Since the support in the Ostrowski numeration system of the integer  $(q_{c_{k+1}}-q_{c_k})/2$  for $k\geq 1$ contains the index $c_{k}+1$, we have $\Lambda_{\rho}(c_{k}-1)=c_k-1$, and we compute, with $k\geq 2$,
\begin{align*}
	\displaystyle \Psi_{c_{k-1}}(q_{c_k}-2- \Psi_{c_k}(\rho))&=\Psi_{c_{k-1}}\left(q_{c_k}-2-\sum_{l = 1}^{k-1}\frac{q_{c_{l+1}}-q_{c_l}}{2}\right) =\Psi_{c_{k-1}}\left(q_{c_k}-2-\frac{q_{c_k}-q_{c_1}}{2}\right) \\ &=\Psi_{c_{k-1}}\left(\frac{q_{c_k}-q_{c_1}}{2}+q_{c_1}-2\right)=\frac{q_{c_{k-1}}-q_{c_1}}{2}+q_{c_1}-2 \\ &=q_{c_1}-2 + \sum_{l=1}^{k-2}\frac{q_{c_{l+1}}-q_{c_l}}{2}
\end{align*}
in view of the support in the Ostrowski numeration system of the involved integers. This shows that the $\alpha$-numbers constructed are equivalent to their complement.

To see that they are the only ones, we use a known result as reference, stating that a sturmian word always have exactly one palindromic factor of any given even length, and exactly two palindromic factors of odd length, according to \cite{doubray}. For all $n$, there exists exactly one factor of $c_\alpha$ of length $2n$ that writes in the form $\widetilde{v_n}v_n$ where $v_n$ is a finite word. By unicity, these factors $(v_n)_{n\geq 1}$ are prefix of one another and hence define a sturmian word $x$ such that the bi-infinite word $\widetilde{x}\cdot x$ is a sturmian orbit. In view of our results in theorem \ref{complementairetheoreme} about the complement of an $\alpha$-number, the formal intercept associated to $x$ is equivalent to its complement. A similar argument applies for the set of palindromic factors of $c_\alpha$ of odd length. In fact, our constructions allow the explicit determination and construction of the formal intercepts associated to the sturmian words $x$ such that there exists a finite word $u$ such that the bi-infinite word  $\widetilde{x}\cdot u \cdot x$ is a sturmian orbit.
\end{proof}

We aim now to generalise this construction, but for the operation of division by an integer $N\geq 2$, the construction above being the division by $N=2$. Although we keep the path of the construction above, we have to keep in mind that we do not have any result about factorisations and complement of $\alpha$-numbers in this situation.

\begin{lem}\label{controle}
Let $(b_i)_{i=N+1}^{M+1}$ be a sequence of non-negative integers such that $b_{i+1}\leq a_{i+1}$ for all $N\leq i \leq M$. Then the integer
\begin{center}
	$\displaystyle n=\sum_{i=N}^Mb_{i+1}q_i$
\end{center} 
admits an expansion in the Ostrowski numeration system whose support is contained in the integer interval $[N,M+1]$, and the coefficient attached to $q_{M+1}$ equals at most 1.
\end{lem}

\begin{proof}
The upper bound on the support and on the value of the coefficient attached to $q_{M+1}$ comes from the inequality
\begin{center}
	$\displaystyle n=\sum_{i=N}^Mb_{i+1}q_i \leq \sum_{i=N}^Ma_{i+1}q_i= q_{M+1}+q_M-(q_{N}+q_{N-1})< q_{M+1}+q_M$
\end{center}
which implies that the highest possible term cannot be greater than $q_{M+1}$.

For the lower bound on the support, we proceed by induction on the number of indexes $N\leq i \leq M$ such that $b_{i+1}=a_{i+1}$. If there are none, then $n$ is already written in the Ostrowski numeration system and we are done. Otherwise, we consider the largest of these indexes, denoted by $i_0$, for which we have by definition $b_{i_0+1}=a_{i_0+1}$.

If $b_{i_0}=0$, then we can apply the induction hypothesis to the integer $\sum_{i=N}^{i_0-2}b_{i+1}q_i$, whose support is then contained in the integer interval $[N,i_0-1]$, with a coefficient in front of $q_{i_0-1}$ that equals at most 1. If this coefficient is zero, then we are done, and if it equals 1, then the cancellation $q_{i_0+1}=a_{i_0+1}q_{i_0}+q_{i_0-1}$ happens, and the coefficients appearing in front of $q_{i_0+1}$ and $q_{i_0}$ are respectively below $a_{i_0+2}$ and zero, which implies that the final expansion satisfies the Ostrowski conditions and the desired properties.

In the case where $b_{i_0}\neq0$, then we can directly proceed to the cancellation $q_{i_0+1}=a_{i_0+1}q_{i_0}+q_{i_0-1}$, and we get
\begin{center}
	$\displaystyle n= \sum_{i=i_0+2}^Mb_{i+2}q_i+(b_{i_0+2}+1)q_{i_0+1}+(b_{i_0}-1)q_{i_0-1}+\sum_{i=N}^{i_0-2}b_{i+1}q_i$
\end{center}
and $b_{i_0+2}+1\leq a_{i_0+2}$  by definition of $i_0$. By induction hypothesis, the integer $(b_{i_0}-1)q_{i_0-1}+\sum_{i=N}^{i_0-2}b_{i+1}q_i$ has a support in the Ostrowski numeration system is included in the integer interval $[N,i_0]$, and its coefficient in front of $q_{i_0}$ equals at most 1. The cancellation $q_{i_0+2}=a_{i_0+2}q_{i_0+1}+q_{i_0}$ can still happen, but in view of the definition of $i_0$, we obtain a support in the Ostrowski numeration system of the integer $n$ that is included in  $[N,M+1]$.
\end{proof}

We consider, for any integer $a$, the matrix
\begin{center}
$A_a=\begin{pmatrix}
	a & 1 \\
	1 & 0
\end{pmatrix}$
\end{center}
so that the continuants $(q_n)_{n\geq -1}$ are bounded to the relation
\begin{align*}
	\begin{pmatrix}
	q_{n+1} \\
	q_n
\end{pmatrix}&= A_{a_{n+1}}
\begin{pmatrix}
	q_{n} \\
	q_{n-1}
\end{pmatrix}\\&= A_{a_{n+1}}A_{a_{n}}A_{a_{n-1}}\cdots A_{a_{1}}\begin{pmatrix}
	1 \\
	0
\end{pmatrix}.
\end{align*}

We fix now an integer $N\geq 2$, and we are going to study the sequence of continuants taken modulo $N$. For $k$ an integer, we denote by $\overline{k}$ its image in $\mathbb{Z}/N\mathbb{Z}$, and we set
\begin{center}
	$A_{\overline{k}}$
\end{center}
the reduction of the matrix $A_k$ in $\mathbb{Z}/N\mathbb{Z}$. We consider the group $G$ generated by the matrices $A_{\overline{k}}$ in $GL_2(\mathbb{Z}/N\mathbb{Z})$. According to the relations
\begin{center}
	$A_0A_1A_{-1}A_0A_{-1}^{-1}A_0=\begin{pmatrix}
	0 & -1 \\
	1 & 0
\end{pmatrix}$ \quad and \quad $A_1A_0=\begin{pmatrix}
	1 & 1 \\
	0 & 1
\end{pmatrix}$
\end{center}
those two matrices being, when seen with coefficients in $\mathbb{Z}$, classical generators of $SL_2(\mathbb{Z})$, and according to the surjectivity of the reduction map $SL_2(\mathbb{Z})\longrightarrow SL_2(\mathbb{Z}/N\mathbb{Z})$ (see the first chapter of \cite{diamond}), added to the fact that the matrices $A_{\overline{k}}$ have determinant $-1$, we see that
\begin{center}
	$\displaystyle G=\left\langle A_{\overline{k}}\right\rangle_{\overline{k}\in \mathbb{Z}/N\mathbb{Z}}=\{A\in GL_2(\mathbb{Z}/N\mathbb{Z}) \ | \ \text{det}A=\pm 1\}$.
\end{center}

We consider the following graph, who encodes the action of $G$ endowed with the generators $(A_{\overline{k}})_{\overline{k}\in \mathbb{Z}/N\mathbb{Z}}$ on the orbits of its natural action on the column vector $\begin{pmatrix}	1 \\ 0 \end{pmatrix}$. It is the graph with set of vertices $V$ as the set of column vectors $\begin{pmatrix}	u \\ v \end{pmatrix}$ such that there exists a matrix $A$ of $G$ such that $\begin{pmatrix}	u \\ v \end{pmatrix}=A\begin{pmatrix}	1 \\ 0 \end{pmatrix}$, and whose directed edges link two column vectors related by a generator of $G$ from the family $(A_{\overline{k}})_{\overline{k}\in \mathbb{Z}/N\mathbb{Z}}$ :
\begin{center}
	$\begin{pmatrix}	u_1 \\ v_1 \end{pmatrix} \stackrel{A_{\overline{k}}}{\longrightarrow} \begin{pmatrix}	u_2 \\ v_2 \end{pmatrix}$ \quad if and only if \quad $\begin{pmatrix}	u_2 \\ v_2 \end{pmatrix}=A_{\overline{k}} \begin{pmatrix}	u_1 \\ v_1 \end{pmatrix}$.
\end{center}

This graph, obviously finite, is part of the setup for the proof of our next (and last) result. It is the Cayley graph of the action of $G$ on $V$ with respect to the generators $(A_{\overline{k}})_{\overline{k}\in \mathbb{Z}/N\mathbb{Z}}$. This graph allows us to use the point of view of automatons to our situation. We consider the alphabet $\mathcal{A}=\mathbb{Z}/N\mathbb{Z}$, and the graph presented above as an automaton whose initial state and final state equal both the column vector $\begin{pmatrix}	1 \\ 0 \end{pmatrix}$. We consider the set $\Gamma$ of words accepted by this automaton :
\begin{center}
	$\displaystyle \Gamma=\left\{ u=u_1u_2u_3\cdots u_n \in  (\mathbb{Z}/N\mathbb{Z})^+ \ \left| \ A_{\widetilde{u}}\begin{pmatrix}	1 \\ 0 \end{pmatrix} =A_{u_n}A_{u_{n-1}}\cdots A_{u_1} \begin{pmatrix}	1 \\ 0 \end{pmatrix} = \begin{pmatrix}	1 \\ 0 \end{pmatrix} \right.\right\}$
\end{center}
where, for a finite word $v=v_1v_2\cdots v_n$ over $\mathcal{A}_N=\mathbb{Z}/N\mathbb{Z}$,
\begin{center}
	$A_{v}=A_{v_1}A_{v_2}\cdots A_{v_n}$.
\end{center}
Once this setup is cleared we can state and proof the following result.

\begin{theo}\label{torsion}
Let $\alpha$ be a slope of continuants $(q_n)_{n\geq -1}$. Then for all $N\geq 2$, there exists a rank $n_0$ from which for all $n\geq n_0$ there exists an integer $k\geq 2$ such that $N$ divides $q_{n+k}-q_n$ and such that the quotient $(q_{n+k}-q_n)/N$ has a support in the Ostroski number system satisfying :
\begin{center}
	$\displaystyle \textup{Supp} \left(\frac{q_{n+k}-q_n}{N}\right)\subset ]n,n+k[$.
\end{center}
\end{theo}

\begin{proof}
We denote by $(a_i)_{i\geq 1}$ the sequence of partial quotients of the slope $\alpha$, and set $y=\overline{a_1}\overline{a_2}\overline{a_3}\cdots$, which is an infinite word over $\mathcal{A}_N=\mathbb{Z}/N\mathbb{Z}$.

We consider the maps $R_N : \mathbb{N} \longrightarrow [0, N[$, who associate to an integer $k$ the remainder of the Euclidean division of $k$ by $N$, so that 
\begin{center}
	$\displaystyle k=\left\lfloor \frac{k}{N}\right\rfloor N+R_N(k)$.
\end{center}
We are going to use this notation in parallel of the notation $\overline{k}$. We give priority to the notation $R_N(k)$ when we are working with positive integers and consider Euclidean divisions, and we give priority to the notation $\overline{k}$ when we use the ring structure of $\mathbb{Z}/N\mathbb{Z}$, or group-theoretic arguments applied to matrices.

Let $n\geq 1$ be an integer, that we can consider as arbitrary large for the following computations. We write :
\begin{align*}
q_n &=a_nq_{n-1}+q_{n-2} \\ &= (\left\lfloor a_n /N \right\rfloor N+R_N(a_n))q_{n-1}+q_{n-2} \\ &= (a_n- R_N(a_n))q_{n-1}+ (a_{n-1}R_N(a_n)+1)q_{n-2}+ R_N(a_n)q_{n-3} \\ &= (a_n- R_N(a_n))q_{n-1}+ (a_{n-1}R_N(a_n)+1- R_N(a_{n-1}R_N(a_n)+1))q_{n-2}+ \\ & \quad (R_N(a_n)+a_{n-2}R_N(a_{n-1}R_N(a_n)+1))q_{n-3}+R_N(a_{n-1}R_N(a_n)+1)q_{n-4} \\ &= (a_n- R_N(a_n))q_{n-1}+ (a_{n-1}R_N(a_n)+1- R_N(a_{n-1}R_N(a_n)+1))q_{n-2}+ \\ & \quad (R_N(a_n)+a_{n-2}R_N(a_{n-1}R_N(a_n)+1)-R_N(R_N(a_n)+a_{n-2}R_N(a_{n-1}R_N(a_n)+1))q_{n-3}+\cdots
\end{align*}
so that if we consider the sequence $(r_i)_{i= 0}^{n-1}$ defined by the recurrence :
\begin{center}
	$r_0=1$, \quad $r_1=R_N(a_n)$, \quad  and \quad $r_{i+1}=R_N(a_{n-i}r_i+r_{i-1})$,
\end{center}
we have, for all $1\leq k \leq n-2$, the relation :
\begin{align*}
	\displaystyle q_n &=(a_n-r_1)q_{n-1}+\sum_{i=1}^{k-1}(a_{n-i}r_i+r_{i-1}-r_{i+1})q_{n-i-1} +(a_{n-k}r_{k}+r_{k-1})q_{n-k-1}+ r_{k}q_{n-k-2} \\ &= (a_n-r_1)q_{n-1}+\sum_{i=1}^{k}(a_{n-i}r_i+r_{i-1}-r_{i+1})q_{n-i-1} +r_{k+1}q_{n-k-1}+ r_{k}q_{n-k-2} \\ &= N \left\lfloor \frac{a_n}{N} \right\rfloor q_{n-1}+ \sum_{i=1}^{k}N \left\lfloor \frac{a_{n-i}r_i+r_{i-1}}{N}\right\rfloor q_{n-i-1} + r_{k+1}q_{n-k-1}+ r_{k}q_{n-k-2}.
\end{align*}
the sequence $(r_k)_{k=0}^{n-1}$, seen in $\mathbb{Z}/N\mathbb{Z}$, satisfies the relation :
\begin{align*}
	\begin{pmatrix}
	r_{k+1} \\
	r_k
\end{pmatrix} & = A_{\overline{a_{n-k}}}
\begin{pmatrix}
	r_{k} \\
	r_{k-1}
\end{pmatrix}\\&= A_{\overline{a_{n-k}}}A_{\overline{a_{n-k+1}}}A_{\overline{a_{n-k+2}}}\cdots A_{\overline{a_{n}}}\begin{pmatrix}
	1 \\
	0
\end{pmatrix}.
\end{align*}

We are going to show that we can find an integer $1\leq k \leq n-2$ such that $\begin{pmatrix}	r_{k+1} \\ r_k \end{pmatrix} = \begin{pmatrix}	0 \\ 1 \end{pmatrix}$, which is equivalent to
\begin{align*}
	A_{\overline{a_{n-k-1}}}
\begin{pmatrix}
	r_{k+1} \\
	r_{k}
\end{pmatrix}&=\begin{pmatrix}	r_{k+2} \\ r_{k+1} \end{pmatrix} \\&= \begin{pmatrix}	1 \\ 0 \end{pmatrix}.
\end{align*}
To show the existence of such a $k$, we use the point of view given by automatons presented above the statement of the theorem. The existence of such a $k$ is equivalent to the relation :
\begin{center}
	$A_{\overline{a_{n-k-1}}}A_{\overline{a_{n-k}}}A_{\overline{a_{n-k+1}}}A_{\overline{a_{n-k+2}}}\cdots A_{\overline{a_{n}}}\begin{pmatrix}
	1 \\
	0
\end{pmatrix}=\begin{pmatrix}
	1 \\
	0
\end{pmatrix}$.
\end{center} 
We then consider the sequence of column vectors :
\begin{center}
	$\displaystyle \left(A_{\mathbb{P}_n(y)}\begin{pmatrix}
	1 \\
	0
\end{pmatrix}\right)_{n\geq 1}=\left(A_{\overline{a_{1}}}A_{\overline{a_{2}}}A_{\overline{a_{3}}}\cdots A_{\overline{a_{n}}}\begin{pmatrix}
	1 \\
	0
\end{pmatrix} \right)_{n\geq 1}$
\end{center}
which is a sequence whose values lie in a finite set (with cardinality bounded above by $N^2$). There exists a rank $n_1\geq1$, from where this sequence only takes values that are attained infinitely many times. That is, there exists a rank $n_1\geq 1$, such that for all $n\geq n_0$, there exists an integer $m>n$ such that
\begin{align*}
	\displaystyle A_{\mathbb{P}_n(y)}\begin{pmatrix}
	1 \\
	0
\end{pmatrix}&=A_{\overline{a_{1}}}A_{\overline{a_{2}}}A_{\overline{a_{3}}}\cdots A_{\overline{a_{n}}}\begin{pmatrix}
	1 \\
	0
\end{pmatrix} \\&= A_{\mathbb{P}_{m}(y)}\begin{pmatrix}
	1 \\
	0
\end{pmatrix}\\&= A_{\overline{a_{1}}}A_{\overline{a_{2}}}A_{\overline{a_{3}}}\cdots A_{\overline{a_{m}}}\begin{pmatrix}
	1 \\
	0
\end{pmatrix}
\end{align*}
from where we derive, since the matrices $A_{u}$ are invertibles,
\begin{center}
	$\begin{pmatrix}
	1 \\
	0
\end{pmatrix}= A_{\overline{a_{n+1}}}A_{\overline{a_{n+2}}}A_{\overline{a_{n+3}}}\cdots A_{\overline{a_{m}}}\begin{pmatrix}
	1 \\
	0
\end{pmatrix}$.
\end{center}
This shows the existence of a rank $n_0\geq 1$ such that for all $n\geq n_0$, there exists an integer $0\leq k \leq n-1$ such that
	\begin{center}
	$A_{\overline{a_{n-k-1}}}A_{\overline{a_{n-k}}}A_{\overline{a_{n-k+1}}}A_{\overline{a_{n-k+2}}}\cdots A_{\overline{a_{n}}}\begin{pmatrix}
	1 \\
	0
\end{pmatrix}=\begin{pmatrix}
	1 \\
	0
\end{pmatrix}$
\end{center} 
and we have seen that this implies $\begin{pmatrix}	r_{k+1} \\ r_k \end{pmatrix} = \begin{pmatrix}	0 \\ 1 \end{pmatrix}$.

In this case, we have the relation
\begin{center}
	$\displaystyle q_n= N \left\lfloor \frac{a_n}{N} \right\rfloor q_{n-1}+ \sum_{i=1}^{k}N \left\lfloor \frac{a_{n-i}r_i+r_{i-1}}{N}\right\rfloor q_{n-i-1} + q_{n-k-2}$
\end{center}
and since the integers $(r_i)_{i=0}^{k}$ belong to the set $[0,N[$, we can bound from above each coefficients attached to the continuants by
\begin{align*}
	\displaystyle \left\lfloor \frac{a_{n-i}r_i+r_{i-1}}{N}\right\rfloor &\leq \frac{a_{n-i}r_i+r_{i-1}}{N} \\&< a_{n-i}+1.
\end{align*}
This inequality being true among integers, we get
\begin{center}
	$\displaystyle \left\lfloor \frac{a_{n-i}r_i+r_{i-1}}{N}\right\rfloor \leq a_{n-i}$
\end{center}
and we are in the proper conditions to apply lemma \ref{controle}. We deduce that the integer
\begin{center}
	$\displaystyle \frac{q_n- q_{n-k-2}}{N}= \left\lfloor \frac{a_n}{N} \right\rfloor q_{n-1}+ \sum_{i=1}^{k} \left\lfloor \frac{a_{n-i}r_i+r_{i-1}}{N}\right\rfloor q_{n-i-1}$
\end{center}
has a support in the Ostrowski numeration system contained in the integer interval $[n-k-1,n]$. In fact, given the trivial upper bound
\begin{center}
	$\displaystyle \frac{q_n- q_{n-k-2}}{N}< q_n$
\end{center}
we can conclude that this support belongs to the set $]n-k-2,n[$.

Our arguments apply to two integers $n\leq m$ such that
\begin{center}
	$\displaystyle A_{\mathbb{P}_n(y)}\begin{pmatrix}
	1 \\
	0
\end{pmatrix}= A_{\mathbb{P}_{m}(y)}\begin{pmatrix}
	1 \\
	0
\end{pmatrix}$,
\end{center}
and we have seen that there exists $n_0\geq 1$ such that for all $n\geq n_0$ there is an integer $m\geq n$ such that the equality above is satisfied, and this allows us to conclude.
\end{proof}

In the case of the Fibonacci sequence, this result is a generalisation of the following formulas :
\begin{center}
	$F_{n+3}-F_n=2F_{n+1}$, \quad $F_{n+8}-F_n=3(F_{n+5}+F_{n+3})$ \quad and \quad $F_{n+6}-F_n=4F_{n+3}$
\end{center}
or
\begin{center}
$F_{n+20}-F_n=5(F_{n+16}+F_{n+13}+F_{n+11}+F_{n+9}+F_{n+6}+F_{n+3})$,
\end{center}
where the numbers appearing on the right-hand side are written in the Ostrowski numeration system.

This results points towards the existence of a summation process on $\alpha$-numbers. Indeed, this result should aim to considerations of sums of the form
\begin{center}
	$\displaystyle \rho =\sum_{k=1}^{+\infty}\frac{q_{c_{k+1}}-q_{c_k}}{N}$
\end{center}
where the integers within the sum have disjoint supporting intervals in the Ostrowski numeration system. A summation process on $\alpha$-number should then provide $N\rho=\sum_{k=1}(q_{c_{k+1}}-q_{c_k})$ which should define a formal intercept of zero equivalence class, hence the interpretation of these relations as torsion relations.

\bigskip

\textit{These results are part of the Ph.D results of the author. Thanks : University Lyon 1, University Lille 1, Harbin Institute of Technology (Shenzhen).}

\end{document}